\documentclass[12pt]{article}
\usepackage{amsmath,amsfonts,amssymb,amsthm}
\usepackage{epsfig}
\usepackage{comment}

\theoremstyle{plain}
\newtheorem{theorem}{Theorem}[section]
\newtheorem{corollary}[theorem]{Corollary}
\newtheorem{lemma}[theorem]{Lemma}
\newtheorem{proposition}[theorem]{Proposition}

\theoremstyle{definition}

\newtheorem*{openproblem}{Open Problem}

\newtheorem{example}[theorem]{Example}
\newtheorem{remark}[theorem]{Remark}

\numberwithin{equation}{section}

\newcommand{\R}{{\mathbb R}}
\newcommand{\N}{{\mathbb N}}

\providecommand{\vint}[1]{\mathchoice
          {\mathop{\vrule width 5pt height 3 pt depth -2.5pt
                  \kern -9pt \kern 1pt\intop}\nolimits_{\kern -5pt{#1}}}
          {\mathop{\vrule width 5pt height 3 pt depth -2.6pt
                  \kern -6pt \intop}\nolimits_{\kern -3pt{#1}}}
          {\mathop{\vrule width 5pt height 3 pt depth -2.6pt
                  \kern -6pt \intop}\nolimits_{\kern -3pt{#1}}}
          {\mathop{\vrule width 5pt height 3 pt depth -2.6pt
                  \kern -6pt \intop}\nolimits_{\kern -3pt{#1}}}}

\newcommand{\eps}{\varepsilon}
\newcommand{\loc}{{\mbox{\scriptsize{loc}}}}

\newcommand{\cH}{\mathcal{H}}

\newcommand{\BV}{\mathrm{BV}}
\newcommand{\liploc}{\mathrm{Lip}_{\mathrm{loc}}}

\DeclareMathOperator{\Mod}{Mod}
\DeclareMathOperator{\capa}{Cap}

\DeclareMathOperator{\dist}{dist}

\DeclareMathOperator{\Lip}{Lip}

\DeclareMathOperator{\supp}{supp}

\def\XXint#1#2#3{{\setbox0=\hbox{$#1{#2#3}{\int}$}
\vcenter{\hbox{$#2#3$}}\kern-.5\wd0}}

\begin{document}
\title{Fine properties and a notion of quasicontinuity for $\BV$ functions on metric spaces
\footnote{
\hfill \break {\it Keywords\,}: bounded variation, metric measure space, Poincar\'e inequality, quasicontinuity, jump set, capacity.
}}
\author{Panu Lahti \\
Mathematical Institute, University of Oxford,\\ Andrew Wiles Building,\\
Radcliffe Observatory Quarter, Woodstock Road,\\
Oxford, OX2 6GG.\\
E-mail: {\tt lahti@maths.ox.ac.uk}\\
\\
Nageswari Shanmugalingam \\
Department of Mathematical Sciences,\\
P.O. Box 210025, University of
Cincinnati,\\
Cincinnati, OH 45221--0025, U.S.A.\\
\noindent E-mail: {\tt shanmun@uc.edu} }
\maketitle
\vspace{-0.5cm}
\begin{abstract} 
On a metric space equipped with a doubling measure supporting a Poincar\'e inequality, we 
show that given a BV function, discarding a set of small $1$-capacity makes the 
function continuous outside its jump set and ``one-sidedly" continuous in its jump set.
We show that such a property implies, in particular, that the measure theoretic boundary of a set of finite perimeter separates the measure theoretic interior of the set from its measure theoretic exterior, both in the sense of the subspace topology outside sets of small $1$-capacity, and in the sense of $1$-almost every curve.
\end{abstract}

\section{Introduction}

Sobolev functions in Euclidean spaces are known to be quasicontinuous. This result holds also in the metric 
setting: if the measure on the metric space is doubling and supports a $(1,1)$-Poincar\'e inequality, then for every
Newton-Sobolev function $u\in N^{1,1}(X)$ there exists an open set $G\subset X$ of small $1$-capacity 
such that the restriction $u|_{X\setminus G}$ is continuous, see for example~\cite{BBS1}. For $p>1$ one can 
even remove the requirement
that the metric space support a $(1,p)$-Poincar\'e inequality.
This follows from the fact that Lipschitz functions are dense in $N^{1,p}(X)$,
see~\cite{ACdM}, together with the fact that
density of Lipschitz functions implies quasicontinuity of $N^{1,p}$-functions.

Such a quasicontinuity property fails for functions of bounded variation, or $\BV$ functions.
From~\cite[Theorem 4.3, Theorem 5.1]{HaKi} (see also~\cite{HS, KinKST10, KKST3})
we know that a set has small $1$-capacity if and only if its codimension $1$ Hausdorff content $\mathcal{H}_R$, for 
any fixed $R>0$, is small. However, $\BV$ 
functions can have jump sets with $\mathcal{H}_R$-measure bounded away from $0$, 
and it is not possible to enclose such
sets within sets of small $1$-capacity.

It is known that a $\BV$ function coincides with a Lipschitz function outside sets of small \emph{measure}, see 
e.g.~\cite[p. 252]{EvaG92} and~\cite[Proposition 4.3]{KKST2}. For spaces $\BV_k(\R^n)$ of higher order $\BV$ functions, 
with $k\in\N$, Lusin-type approximations by means of differentiable functions outside sets of small 
$1$-capacity are given in~\cite[Theorem 6.2]{BKK}. However, even in the Euclidean setting, little appears to be known about 
the behavior of (first-order) $\BV$ functions outside sets of small $1$-capacity.
The goal of the current paper is to show a weak notion of quasicontinuity for $\BV$ functions, involving continuity 
outside the jump set and ``one-sided" continuity up to the jump set.

In what follows, $X$ is a metric space equipped with a metric $d$ and a doubling Borel regular outer measure $\mu$ that supports  $(1,1)$-Poincar\'e inequality.
Definitions and notation will be discussed systematically in Section~\ref{sec:preliminaries}.
The jump set of a function $u\in\BV(X)$ is defined as
\[
S_{u}:=\{x\in X:\, u^{\wedge}(x)<u^{\vee}(x)\},
\]
where $u^{\wedge}(x)$ and $u^{\vee}(x)$ are the lower and upper approximate limits of $u$ defined as
\begin{equation*}
u^{\wedge}(x):
=\sup\left\{t\in\overline\R:\,\lim_{r\to 0^+}\frac{\mu(B(x,r)\cap\{u<t\})}{\mu(B(x,r))}=0\right\}
\end{equation*}
and
\begin{equation*}
u^{\vee}(x):
=\inf\left\{t\in\overline\R:\,\lim_{r\to 0^+}\frac{\mu(B(x,r)\cap\{u>t\})}{\mu(B(x,r))}=0\right\}.
\end{equation*}

It was shown in~\cite[Theorem 5.3]{AMP} that
$\mathcal{H}$ is a $\sigma$-finite measure on $S_u$.
Furthermore, from~\cite[Theorem 5.4]{A1} we know that there is a
number $0<\gamma\le 1/2$ such that if $E\subset X$ is a set of finite perimeter (that is, $\chi_E\in\BV(X)$), then
 the perimeter measure $P(E,\cdot)$ is carried on the set $\Sigma_\gamma E$, which is the collection of points
$x\in X$ for which
\[
 \gamma\le \liminf_{r\to 0^+}\frac{\mu(B(x,r)\cap E)}{\mu(B(x,r))}
   \le \limsup_{r\to 0^+}\frac{\mu(B(x,r)\cap E)}{\mu(B(x,r))}\le 1-\gamma.
\]
Classical results on $\BV$ functions in the Euclidean setting can be formulated in terms of the approximate limits $u^{\wedge}$ and $u^{\vee}$, but in the general metric setting we need to consider a larger number of jump values. The reason for this will be illustrated in Example~\ref{ex:one dimensional space}.
Given $u\in\BV(X)$, we define the functions (jump values) $u^l$, $l=1,\ldots, n:=\lfloor 1/\gamma\rfloor$ (with $\gamma$ 
as above), by
$u^1:=u^{\wedge}$, $u^n:=u^{\vee}$, 
and for $l=2,3,\ldots, n-1$, we set
\[
u^{l}(x):=\sup\left\{t\in\overline{\R}:\,\lim_{r\to 0^+}\frac{\mu(B(x,r)\cap \{u^{l-1}(x)+\delta<u<t\})}{\mu(B(x,r))}=0\ \ \forall\, \delta>0\right\}
\]
provided $u^{l-1}(x)<u^\vee(x)$, and otherwise, we set 
$u^l(x)=u^{\vee}(x)$. We have $u^{\wedge}=u^1\le \ldots \le  u^n=u^{\vee}$.
We also define $\widetilde{u}:=(u^{\wedge}+u^{\vee})/2$.
Note that if $x\in X\setminus S_u$, then $u^1(x)=\ldots=u^n(x)$.

The following theorem, which is the main result of this paper, introduces a notion of quasicontinuity for $\BV$ functions.

\begin{theorem}\label{thm:main result}
Let $u\in\BV(X)$ and let $\eps>0$. Then there exists an open set $G\subset X$ with $\capa_1(G)<\eps$ such that if $y_k\to x$ 
with $y_k,x\in X\setminus G$, then 
\[
\min_{l_2\in \{1,\ldots,n\}} |u^{l_1}(y_k)-u^{l_2}(x)|\to 0
\]
for each $l_1=1,\ldots,n$.
\end{theorem}

In particular, $\widetilde{u}\vert_{X\setminus G}$ is continuous at every $x\in X\setminus (S_u\cup G)$.
The proof of Theorem~\ref{thm:main result} is given in two parts; in
Proposition~\ref{prop:quasicontinuity for BV} we prove continuity outside the jump set, and in Proposition~\ref{prop:one-sided continuity} we prove ``one-sided" continuity up to the jump set. In proving the ``one-sided" continuity, we show that if $x\in S_u\setminus G$, then $X$ can be partitioned into
at most $n^2$ number of sets $(u^{l_1})^{-1}(A^\delta_{l_2}(x))$, defined in \eqref{eq:definition of Als}, such that when the sequence $y_k$ lies
in $(u^{l_1})^{-1}(A^\delta_{l_2}(x))\setminus G$ and converges to $x$, we must have $u^{l_1}(y_k)\to u^{l_2}(x)$.\\

\noindent {\bf Acknowledgment:} Part of the research for this paper was conducted during the visit of the first author to the
University of Cincinnati
in 2015, and he wishes to thank the institution for its kind hospitality. The research of the first author
is partially funded by a grant from the Finnish Cultural Foundation. The research of the second author is partially funded
by the National Science Foundation (U.S.A.) grant~\#DMS-1500440.

\section{Background}\label{sec:preliminaries}

In this section we introduce the necessary definitions and assumptions.

Throughout the paper, $(X,d,\mu)$ is a complete metric space equipped
with a Borel regular outer measure $\mu$ satisfying a doubling property, that is,
there is a constant $C_d\ge 1$ such that
\[
0<\mu(B(x,2r))\leq C_d\,\mu(B(x,r))<\infty
\]
for every ball $B=B(x,r)$ with center $x\in X$ and radius $r>0$. Given a ball $B=B(x,r)$ and $\tau>0$, we denote by
$\tau B$  the ball $B(x,\tau r)$. 
In a metric space, a ball does not necessarily have a unique center 
and radius, but whenever we use the above abbreviation we will consider balls whose center and radii have been pre-specified,
and so no ambiguity arises. 

By iterating the doubling condition, we obtain that there are constants $C\ge 1$ and $Q>0$ such that 
\begin{equation}\label{eq:definition of Q}
\frac{\mu(B(y,r))}{\mu(B(x,R))}\ge C^{-1} \left(\frac{r}{R}\right)^Q
\end{equation}
for every $0<r\le R$ and $y\in B(x,R)$. The choice $Q=\log_2(C_d)$ works, but a smaller value of $Q$ might
satisfy the above condition as well.

In general, $C\ge 1$ will denote a generic constant whose particular value is not important for the purposes of this
paper, and might differ between
each occurrence. When we want to specify that a constant $C$
depends on the parameters $a,b, \ldots,$ we write $C=C(a,b,\ldots)$. Unless otherwise specified, all constants only 
depend on the doubling constant $C_d$ and the constants $C_P,\lambda$ associated
with the Poincar\'e inequality defined below.

Given $x\in X$ and $A_1,A_2\subset X$, we set
\[
\dist(x,A_1):=\inf\{d(x,y) :\, y\in A_1\}, \hskip .2cm \dist(A_1,A_2):=\inf\{d(z,A_1) :\, z\in A_2\}.
\]

A complete metric space with a doubling measure is proper,
that is, closed and bounded sets are compact. Since $X$ is proper, for any open set $\Omega\subset X$
we define $\liploc(\Omega)$ to be the space of
functions that are Lipschitz in every $\Omega'\Subset\Omega$.
Here $\Omega'\Subset\Omega$ means that $\Omega'$ is open and that $\overline{\Omega'}$ is a
compact subset of $\Omega$.
We define other local spaces similarly.

For any set $A\subset X$ and $0<R<\infty$, the restricted spherical Hausdorff content
of codimension $1$ is defined as
\[
\mathcal{H}_{R}(A):=\inf\left\{ \sum_{i=1}^{\infty}
  \frac{\mu(B(x_{i},r_{i}))}{r_{i}}:\,A\subset\bigcup_{i=1}^{\infty}B(x_{i},r_{i}),\,r_{i}\le R\right\}.
\]
We define the above also for $R=\infty$ by requiring $r_i<\infty$. 
The codimension $1$ Hausdorff measure of a set $A\subset X$ is given by
\begin{equation*}
  \mathcal{H}(A):=\lim_{R\rightarrow 0^+}\mathcal{H}_{R}(A).
\end{equation*}

The measure theoretic boundary $\partial^{*}E$ of a set $E\subset X$ is the set of all points $x\in X$
at which both $E$ and its complement have positive upper density, i.e.
\[
\limsup_{r\to 0^+}\frac{\mu(B(x,r)\cap E)}{\mu(B(x,r))}>0\quad\;
  \textrm{and}\quad\;\limsup_{r\to 0^+}\frac{\mu(B(x,r)\setminus E)}{\mu(B(x,r))}>0.
\]

A curve is a rectifiable continuous mapping from a compact interval
into $X$.
The length of a curve $\gamma$
is denoted by $\ell_{\gamma}$. We will assume every curve to be parametrized
by arc-length, which can always be done (see e.g.~\cite[Theorem~3.2]{Hj} or~\cite{AT}).
A nonnegative Borel function $g$ on $X$ is an \emph{upper gradient} 
of an extended real-valued function $u$
on $X$ if for all curves $\gamma$ on $X$, we have
\begin{equation}\label{eq:upper gradient definition}
|u(x)-u(y)|\le \int_\gamma g\,ds,
\end{equation}
where $x$ and $y$ are the end points of $\gamma$. We interpret $|u(x)-u(y)|=\infty$ whenever  
at least one of $|u(x)|$, $|u(y)|$ is infinite. Upper gradients were originally introduced in~\cite{HK}.

Let $\Gamma$ be a family of curves,
and let $1\le p<\infty$. The $p$-modulus of $\Gamma$ is defined
as
\[
\Mod_{p}(\Gamma):=\inf\int_{X}\rho^{p}\, d\mu
\]
where the infimum is taken over all nonnegative Borel functions $\rho$
such that $\int_{\gamma}\rho\,ds\ge 1$ for every $\gamma\in\Gamma$.
If a property fails only for a curve family with $p$-modulus zero,
we say that it holds for $p$-almost every (a.e.) curve.
If $g$ is a nonnegative $\mu$-measurable function on
$X$ and~\eqref{eq:upper gradient definition} holds for $p$-almost every curve, then
$g$ is a $p$\emph{-weak upper gradient} of $u$.

We consider the following norm
\[
\Vert u\Vert_{N^{1,p}(X)}:=\Vert u\Vert_{L^p(X)}+\inf_g\Vert g\Vert_{L^p(X)},
\]
with the infimum taken over all upper gradients $g$ of $u$. 
The substitute for the Sobolev space $W^{1,p}(\R^n)$ in the metric setting is the following Newton-Sobolev space 
\[
N^{1,p}(X):=\{u:\|u\|_{N^{1,p}(X)}<\infty\}/{\sim},
\]
where the equivalence relation $\sim$ is given by $u\sim v$ if and only if 
\[
\Vert u-v\Vert_{N^{1,p}(X)}=0.
\]
Similarly, we can define $N^{1,p}(\Omega)$ for any open set $\Omega\subset X$.
For more on Newton-Sobolev spaces, we refer to~\cite{S, HKST, BB}.

Next we recall the definition and basic properties of functions
of bounded variation on metric spaces, see \cite{M}. See also e.g. \cite{AFP, Giu84, Zie89} for the classical 
theory in the Euclidean setting.
For $u\in L^1_{\loc}(X)$, we define the total variation of $u$ on $X$ to be 
\[
\|Du\|(X):=\inf\Big\{\liminf_{i\to\infty}\int_X g_{u_i}\,d\mu:\, u_i\in \Lip_{\loc}(X),\, u_i\to u\textrm{ in } L^1_{\loc}(X)\Big\},
\]
where each $g_{u_i}$ is an upper gradient of $u_i$.
We say that a function $u\in L^1(X)$ is \emph{of bounded variation}, 
and denote $u\in\BV(X)$, if $\|Du\|(X)<\infty$. 
A measurable set $E\subset X$ is said to be of \emph{finite perimeter} if $\|D\chi_E\|(X)<\infty$. The \emph{perimeter}
of $E$ in $X$ is denoted by
\[
P(E,X):=\|D\chi_E\|(X).
\]
By replacing $X$ with an open set $\Omega\subset X$ in the definition of the total variation, we can define $\|Du\|(\Omega)$.
The $\BV$ norm is given by
\begin{equation*}
\Vert u\Vert_{\BV(\Omega)}:=\Vert u\Vert_{L^1(\Omega)}+\Vert Du\Vert(\Omega).
\end{equation*}
It was shown in~\cite[Theorem~3.4]{M} that for $u\in\BV(X)$, $\Vert Du\Vert$ is the restriction to the class of 
open sets of a finite Radon measure defined on the
class of all subsets of $X$. This outer measure is obtained from the map $\Omega\mapsto\Vert Du\Vert(\Omega)$ on open sets
$\Omega\subset X$ via the standard Carath\'eodory construction. Thus, 
for an arbitrary set $A\subset X$, 
\[
\|Du\|(A):=\inf\bigl\{\|Du\|(U):\, A\subset U\subset X 
\text{ with }U\text{ open}\bigr\}.
\]
Similarly, if $u\in\BV(\Omega)$, then $\|Du\|(\cdot)$ is a finite Radon measure on $\Omega$.

We have the following coarea formula from~\cite[Proposition 4.2]{M}: if $F\subset X$ is a Borel set and 
$u\in \BV(X)$, then
\begin{equation}\label{eq:coarea}
\|Du\|(F)=\int_{-\infty}^{\infty}P(\{u>t\},F)\,dt.
\end{equation}
In particular, the map $t\mapsto P(\{u>t\},F)$ is Lebesgue measurable on $\R$.

We will assume that $X$ supports a $(1,1)$-Poincar\'e inequality,
meaning that there are constants $C_P>0$ and $\lambda \ge 1$ such that for every
ball $B(x,r)$, for every locally integrable function $u$ on $X$,
and for every upper gradient $g$ of $u$, we have 
\[
\vint{B(x,r)}|u-u_{B(x,r)}|\, d\mu 
\le C_P r\vint{B(x,\lambda r)}g\,d\mu,
\]
where 
\[
u_{B(x,r)}:=\vint{B(x,r)}u\,d\mu :=\frac 1{\mu(B(x,r))}\int_{B(x,r)}u\,d\mu.
\]
By applying the Poincar\'e inequality to approximating Lipschitz functions in the definition of the total variation, 
we get the following $(1,1)$-Poincar\'e inequality for $\BV$ functions. There exists a constant $C$
such that for every ball $B(x,r)$ and every 
$u\in L^1_{\loc}(X)$, we have
\begin{equation}\label{eq:poincare for BV}
\vint{B(x,r)}|u-u_{B(x,r)}|\,d\mu
\le Cr\, \frac{\Vert Du\Vert (B(x,\lambda r))}{\mu(B(x,\lambda r))}.
\end{equation}

Sets of measure zero play a fundamental role in the theory of $L^p$ spaces. In potential theory sets of measure
zero can be too large to be discarded; a finer measure of the smallness of a set is needed. 
For $1\le p<\infty$, the $p$-capacity of a set $A\subset X$ is given by
\begin{equation}\label{eq:cap-def}
 \capa_p(A):=\inf \Vert u\Vert_{N^{1,p}(X)},
\end{equation}
where the infimum is taken over all functions $u\in N^{1,p}(X)$ such that $u\ge 1$ in a neighborhood of $A$; we can further restrict the
class of functions $u$ by requiring that $0\le u\le 1$ on $X$. It follows from~\cite[Theorem~4.3, Theorem~5.1]{HaKi} that
$\capa_1(E)=0$ if and only if $\mathcal{H}(E)=0$.  

Given a set $E\subset X$ of finite perimeter, for $\mathcal H$-a.e. $x\in \partial^*E$ we have
\begin{equation}\label{eq:density of E}
\gamma \le \liminf_{r\to 0^+} \frac{\mu(E\cap B(x,r))}{\mu(B(x,r))} \le \limsup_{r\to 0^+} \frac{\mu(E\cap B(x,r))}{\mu(B(x,r))}\le 1-\gamma
\end{equation}
where $\gamma \in (0,1/2]$ only depends on the doubling constant and the constants in the Poincar\'e inequality, 
see~\cite[Theorem 5.4]{A1}. We denote the set of all such points by $\Sigma_\gamma E$.

For a Borel set $F\subset X$ and a set $E\subset X$ of finite perimeter, we know that 
\begin{equation}\label{eq:def of theta}
\Vert D\chi_{E}\Vert(F)=\int_{\partial^{*}E\cap F}\theta_E\,d\mathcal H,
\end{equation}
where $\partial^*E$ is the measure-theoretic boundary of $E$ and
$\theta_E:X\to [\alpha,C_d]$, with $\alpha=\alpha(C_d,C_P,\lambda)>0$, see \cite[Theorem 5.3]{A1} 
and \cite[Theorem 4.6]{AMP}.

The \emph{jump set} of $u\in\BV(X)$ is the set 
\[
S_{u}:=\{x\in X:\, u^{\wedge}(x)<u^{\vee}(x)\},
\]
where $u^{\wedge}(x)$ and $u^{\vee}(x)$ are the lower and upper approximate limits of $u$ defined respectively by
\begin{equation}\label{eq:lower approximate limit}
u^{\wedge}(x):
=\sup\left\{t\in\overline\R:\,\lim_{r\to 0^+}\frac{\mu(B(x,r)\cap\{u<t\})}{\mu(B(x,r))}=0\right\}
\end{equation}
and
\begin{equation}\label{eq:upper approximate limit}
u^{\vee}(x):
=\inf\left\{t\in\overline\R:\,\lim_{r\to 0^+}\frac{\mu(B(x,r)\cap\{u>t\})}{\mu(B(x,r))}=0\right\}.
\end{equation}
We also define the functions $u^l$, $l=1,\ldots,n=\lfloor 1/\gamma\rfloor$, as 
follows: $u^1:=u^{\wedge}$, $u^n:=u^{\vee}$, 
and for $l=2,\ldots,n-1$ we define inductively
\begin{equation}\label{eq:definition of the n limits}
u^{l}(x):=\sup\left\{t\in\overline{\R}:\,\lim_{r\to 0^+}\frac{\mu(B(x,r)\cap \{u^{l-1}(x)+\delta<u<t\})}{\mu(B(x,r))}=0\ \ \forall\, \delta>0\right\}
\end{equation}
provided $u^{l-1}(x)<u^\vee(x)$, and otherwise, we set 
$u^l(x)=u^{\vee}(x)$. 
It can be shown that each $u^l$ is a Borel function,  
and $u^{\wedge}=u^1\le \ldots \le u^n = u^{\vee}$. 

Given the definition of the BV norm, we understand BV functions to be $\mu$-equivalence classes. 
To consider 
questions of continuity, we need to consider the pointwise representatives $u^l$, $l=1,\ldots,n$.
We also use the standard representative
$\widetilde{u}:=(u^{\wedge}+u^{\vee})/2$. 

By \cite[Theorem 5.3]{AMP}, the variation measure of a $\BV$ function can be decomposed into the absolutely 
continuous and singular part, and the latter into the Cantor and jump part, as follows. Given an open set 
$\Omega\subset X$ and $u\in\BV(\Omega)$, we have
\begin{equation}\label{eq:decomposition}
\begin{split}
\Vert Du\Vert(\Omega)
&=\Vert Du\Vert^a(\Omega)+\Vert Du\Vert^s(\Omega)\\
&=\Vert Du\Vert^a(\Omega)+\Vert Du\Vert^c(\Omega)+\Vert Du\Vert^j(\Omega)\\
&=\int_{\Omega}a\,d\mu+\Vert Du\Vert^c(\Omega)
   +\int_{\Omega\cap S_u}\int_{u^{\wedge}(x)}^{u^{\vee}(x)}\theta_{\{u>t\}}(x)\,dt\,d\mathcal H(x)
\end{split}
\end{equation}
where $a\in L^1(\Omega)$ is the density of the absolutely continuous part and the functions $\theta_{\{u>t\}}$ 
are as in~\eqref{eq:def of theta}.

For $R>0$, the restricted maximal function of a function $v\in L_{\loc}^1(X)$ is given by
\[
\mathcal M_R v(x):=\sup_{0<r\le R}\,\vint{B(x,r)}|v|\,d\mu,\qquad x\in X.
\]
The following result will be used a few times. See~\cite[Lemma 4.3, Remark 4.9]{KKST3} for a proof. 
While~\cite{KKST3} makes the extra assumption $\mu(X)=\infty$,
use of this assumption can be avoided by considering $\mathcal{H}_R$ instead of $\mathcal{H}_\infty$.

\begin{lemma}\label{lem:capacity level set maximal function}
There exists $C=C(C_d,C_P,\lambda,R)$ 
such that for every $u\in\BV(X)$ and $t>0$, 
\[
\capa_1(\{\mathcal{M}_R u\ge t\})\le \frac{C \Vert u\Vert_{\BV(X)}}{t}.
\]
\end{lemma}

\section{Discrete convolutions}

In this section we discuss functions in $\BV(U)$ with zero boundary values on $\partial U$, and methods 
of ``mollifying" BV functions in open sets.
For a proof of the following theorem, see \cite[Theorem~6.1]{LS} or \cite{KKST}. 

\begin{theorem}\label{thm:zero extension}
Let $U\subset X$ be an open set and $u\in \BV(U)$. Assume
that $\mathcal H(\partial U)<\infty$. If
\begin{equation}\label{eq:weak zero trace condition}
\lim_{r\to 0^+}\frac{1}{\mu(B(x,r))}\int_{B(x,r)\cap U}|u|\,d\mu=0
\end{equation}
for $\cH$-a.e. $x\in \partial U$, then
the zero extension of $u$ into the whole space $X$, denoted by $\widehat{u}$,
is in $\BV(X)$ with $\Vert D\widehat{u}\Vert(X\setminus U)=0$.
\end{theorem}

The following technical lemma can be proved by a simple covering argument.

\begin{lemma}[{\cite[Lemma~6.4]{LS}}]\label{lem:measure in domain}
Let $U\subset X$ be an open set, let $\nu$ be a finite Radon measure on $U$, and define
\[
  A:=\left\{x\in\partial U :\,\limsup_{r\to 0^+}r \frac{\nu(B(x,r)\cap U)}{\mu(B(x,r))} > 0\right\}.
\]
Then $\cH(A )=0$.
\end{lemma}

In most of the paper, we will work with Whitney type coverings of open sets.
For the construction of such coverings and their properties, see e.g.~\cite[Theorem 3.1]{BBS07} (such coverings
were originally introduced in the Euclidean setting by Whitney in~\cite[Section~8, page~67]{Wh}, and subsequently
extended to more general settings in~\cite[Theorem~III.1.3]{CoWe} and~\cite[Lemma~2.9]{MaSe}).

Given any  open set $U\subset X$ and a scale $R>0$, we can choose a Whitney type covering 
$\{B_j=B(x_j,r_j)\}_{j=1}^{\infty}$ of $U$ such that
\begin{enumerate}
\item for each $j\in\N$,
\begin{equation}\label{eq:radius-dist-est}
r_j= \min\left\{\frac{\dist(x_j,X\setminus U)}{40\lambda},\,R\right\},
\end{equation}
\item for each $k\in\N$, the ball $10\lambda B_k$ meets at most $C_0=C_0(C_d,\lambda)$ 
balls $10\lambda B_j$ (that is, a bounded overlap property holds),
\item  
if $10\lambda B_j$ meets $10\lambda B_k$, then $r_j\le 2r_k$.
\end{enumerate}

Given such a covering of $U$, 
we can take a partition of unity $\{\phi_j\}_{j=1}^{\infty}$ subordinate to this cover, such that $0\le \phi_j\le 1$, each 
$\phi_j$ is a $C/r_j$-Lipschitz function, and $\supp(\phi_j)\subset 2B_j$ for each 
$j\in\N$ (see e.g. \cite[Theorem 3.4]{BBS07}). Finally, we can define a \emph{discrete convolution} $v$ of 
any $u\in L^1_{\loc}(U)$ with respect to the Whitney type covering by
\[
v:=\sum_{j=1}^{\infty}u_{B_j}\phi_j.
\]
In general, $v$ is locally Lipschitz in $U$, and hence belongs to $ L^1_{\loc}(U)$. 
If $u\in L^1(U)$, then $v\in L^1(U)$.

The goal of the next proposition is to show that the discrete convolution $v$ of $u$ has the same boundary values
as $u$, i.e. that $v-u$ has zero boundary values in the sense of Theorem~\ref{thm:zero extension}.

\begin{proposition}
\label{prop:traces for discrete convolutions}
Let $U\subset X$ be an open set, $R>0$, 
and $u\in\BV(U)$. Let $v\in\liploc(U)$ be the discrete 
convolution of $u$ with respect to a Whitney type covering $\{B_j=B(x_j,r_j)\}_{j=1}^{\infty}$ of $U$ 
at scale $R$. Then
\begin{equation}\label{eq:weak trace result for discrete convolutions}
\lim_{r\to 0^+}\frac{1}{\mu(B(x,r))}\int_{B(x,r)\cap U}|v-u|\,d\mu=0
\end{equation}
for $\mathcal H$-a.e. $x\in\partial U$.
\end{proposition}

This proposition was previously given in~\cite[Proposition~6.5]{LS}, but 
we include the proof
here as well since it is simple enough and makes the exposition more self-contained.

\begin{proof}
Fix $x\in\partial U$ and $r>0$. Denote by $I_r$ the set of indices $j\in\N$ for which 
$2B_j\cap B(x,r)\neq \emptyset$. Note that from~\eqref{eq:radius-dist-est} and the fact that $\lambda\ge 1$ it follows that
\[
r_j\le \frac{\dist(2B_j, X\setminus U)}{38\lambda}\le \frac{r}{38\lambda}
\]
for every $j\in I_r$.
Because $\sum_{j\in\N}\phi_j=\chi_U$, we have (in fact, the following holds for any $x\in X\setminus U$)
\begin{equation}\label{eq:discrete convolution boundary value calculation}
\begin{split}
\int_{B(x,r)\cap U}|u-v|\,d\mu
&= \int_{B(x,r)\cap U}\Big|\sum_{j\in I_r}u\phi_j-\sum_{j\in I_r}u_{B_j}\phi_j\Big|\,d\mu\\
&\le \int_{B(x,r)\cap U}\sum_{j\in I_r}\Big|\phi_j(u-u_{B_j})\Big|\,d\mu\\
&\le \sum_{j\in I_r}\,\int_{2B_j}|u-u_{B_j}|\,d\mu\\
&\le \sum_{j\in I_r}\left(\int_{2B_j}|u-u_{2B_j}|\,d\mu+\int_{2B_j}|u_{2B_j}-u_{B_j}|\,d\mu\right)\\
&\le 2C_d \sum_{j\in I_r}\,\int_{2B_j}|u-u_{2B_j}|\,d\mu\\
&\le 4C_d C_P \sum_{j\in I_r}r_j\Vert Du\Vert (2\lambda B_j)\\
&\le C r \Vert Du\Vert (B(x,2r)\cap U).
\end{split}
\end{equation}
In the above, we used the fact that $X$ supports a $(1,1)$-Poincar\'e inequality, and in the last inequality we used the fact that $2\lambda B_j\subset U\cap B(x,2r)$ for all $j\in I_r$, as well as
 the bounded overlap of the dilated Whitney balls $2\lambda B_j$.
Thus by Lemma~\ref{lem:measure in domain}, we have
\[
\lim_{r\to 0^+}\frac{1}{\mu(B(x,r))}\int_{B(x,r)\cap U}|u-v|\,d\mu=0
\]
for $\mathcal H$-a.e. $x\in\partial U$.
\end{proof}

Let $U\subset X$ be an open set, $R>0$, and as above, 
let $v$ be the discrete convolution of a function $u\in \BV(U)$ 
with respect to a Whitney type covering $\{B_j\}_{j\in\N}$ of $U$ at scale $R$. Then $v$ has an upper gradient
\begin{equation}\label{eq:upper gradient of discrete convolution}
g=C\sum_{j=1}^{\infty}\chi_{B_j}\frac{\Vert Du\Vert(5\lambda B_j)}{\mu(B_j)}
\end{equation}
in $U$ (with $C$ depending, as usual, only on the doubling constant and the constants in the 
Poincar\'e inequality),
see e.g. the proof of~\cite[Proposition~4.1]{KKST2}. 
From  the proof of this result it also follows that in a small ball comparable to the size of a Whitney ball, say 
$B=B(x,\min\{\dist(x,X\setminus U)/20\lambda,R\})$, $v$ is Lipschitz with constant, say,
\begin{equation}\label{eq:lipschitz constant of discrete convolution}
C\frac{\Vert Du\Vert(B(x,\min\{\dist(x,X\setminus U)/4,5\lambda R\}))}{\mu(B)}.
\end{equation}
Also, if $V_\epsilon\subset U$, $\epsilon>0$, is any family of 
open subsets of $U$ and every $v_\epsilon$ is a 
discrete convolution of a function $u\in L^1(U)$ with 
respect to a Whitney type covering of $V_\epsilon$ at scale $\epsilon>0$, then
\begin{equation}\label{eq:L1 convergence for discrete convolutions}
\lim_{\epsilon\to 0^+}\Vert v_\epsilon- u\Vert_{L^1(V_\epsilon)}=0,
\end{equation}
as seen by the discussion in the proof of~\cite[Lemma 5.3]{HKT}.

It is often useful to be able to ``mollify" $\BV$ functions in small open sets where e.g. a certain part 
of the variation measure lives. 
Combining the above discussion on discrete convolutions with Theorem~\ref{thm:zero extension} and 
Proposition~\ref{prop:traces for discrete convolutions}, we obtain the following result on such mollifications.

\begin{corollary}\label{cor:gluing discrete convolutions}
Let $U\subset X$ be an open set, and let $u\in\BV(U)$. Assume  that 
$\mathcal H(\partial U)<\infty$. Let each 
$v_i\in\liploc(U)$ be the discrete convolution of $u$ with respect to a Whitney type covering of $U$ at scale $1/i$, $i\in\N$. 
Then $v_i\to u$ in $L^1(U)$, $\Vert Dv_i\Vert(U)\le C\Vert Du\Vert(U)$, and the functions
\[
h_i:=
\begin{cases}
v_i-u &\ \text{in }U,\\
0  &\ \text{in }X\setminus U
\end{cases}
\]
satisfy $h_i\in \BV(X)$ and $\Vert Dh_i\Vert(X\setminus U)=0$.
\end{corollary}

In the above we require
that the boundary of $U$ has finite $\mathcal{H}$-measure. 
However, for the proof of the main theorem of this paper, we need ``mollifications" on arbitrary open sets.
In the following, we extend Corollary~\ref{cor:gluing discrete convolutions} to all open sets.
Recall that $\widetilde{u}:=(u^{\wedge}+u^{\vee})/2$, where the lower and upper approximate limits 
$u^\wedge$, $u^\vee$ were defined in~\eqref{eq:lower approximate limit} and~\eqref{eq:upper approximate limit}.

\begin{theorem}\label{thm:mollifying in an open set}
Let $U\subset X$ be an open set, $u\in\BV(U)$, and $\kappa>0$. 
Then there exists a function $w\in\BV(U)$ satisfying the following:
$\widetilde{w}\in N^{1,1}(U)\cap \liploc(U)$ with an
upper gradient $g$ satisfying $\Vert g\Vert_{L^1(U)}\le C\Vert Du\Vert(U)$; 
$\Vert w-u\Vert_{L^1(U)}\le\kappa$; and the function
\[
h:=
\begin{cases}
w-u &\ \text{in }U,\\
0  &\ \text{in }X\setminus  U
\end{cases}
\]
satisfies $h\in\BV(X)$ and $\Vert Dh\Vert(X\setminus U)=0$.
\end{theorem}

\begin{proof}
The following coarea inequality is known to hold: if $\omega\in\Lip(X)$ is $L$-Lipschitz, then
\begin{equation}\label{eq:coarea inequality}
\int_{\R}\mathcal H(A\cap \omega^{-1}(t))\,dt\le C L\mu(A)
\end{equation}
for any Borel set $A\subset X$, see~\cite[Proposition 3.1.5]{AT} or~\cite{Fed}.
The proof in~\cite{AT} deals with Hausdorff measures $\mathcal{H}^{k-m}$, $\mathcal{H}^m$
and $\mathcal{H}^k$ instead of $\mathcal{H}$, $dt$ and $\mu$; however, their proof
of this result holds in our setting when we replace $\mathcal{H}^k$ with $\mu$,
$\mathcal{H}^{k-m}$ with $\mathcal{H}$, and set $m=1$.
By considering the $1$-Lipschitz functions
\[
\omega_1(y):=\dist(y,X\setminus U)\qquad\textrm{and}\qquad \omega_2(y):=d(y,x)
\]
for a fixed $x\in X$, we can pick open sets 
$U_1\Subset U_2\Subset \ldots\subset U$, defined as
\[
U_i:=\{\omega_1>\alpha_i\ \textrm{and} \  \omega_2<1/\alpha_i\}
\]
for some strictly decreasing sequence $\alpha_i\searrow 0$. Clearly $U=\bigcup_{i\in\N}U_i$, and 
by a suitable choice of the sequence $\alpha_i$, \eqref{eq:coarea inequality} gives
$\mathcal H(\partial U_i)<\infty$ 
for each $i\in\N$.

Fix a scale $R>0$.
For each $i\in\N$, define $v_i$ to be the discrete convolution of $u$ with respect to a Whitney 
type covering $\{B_j^i\}_{j\in\N}$ of $U_i$, at scale $R$.
By~\eqref{eq:L1 convergence for discrete convolutions}
we can choose $R$ to be small enough so that
$\Vert v_i-u\Vert_{L^1(U_i)}\le \kappa$ for each $i\in\N$. Each $v_i$ has an upper gradient $g_i$ 
in $U_i$, defined in~\eqref{eq:upper gradient of discrete convolution}, with 
\begin{equation}\label{eq:upper gradients in the sets Ui}
\Vert g_i\Vert_{L^1(U_i)}\le C\Vert Du\Vert(U_i).
\end{equation}
By Corollary~\ref{cor:gluing discrete convolutions}, the function
\[
h_i=\begin{cases}
v_i-u &\textrm{in }U_i,\\
0   &\textrm{in }X\setminus U_i
\end{cases}
\]
satisfies $h_i\in\BV(X)$ and
$\Vert Dh_i\Vert(X\setminus U_i)=0$. Hence
\begin{align*}
\Vert D h_i\Vert(X)
&\le \Vert Dv_i\Vert(U_i)+\Vert Du\Vert(U_i)
\le \Vert g_i\Vert_{L^1(U_i)}+\Vert Du\Vert( U_i)\\
&\le C\Vert Du\Vert(U_i)\le C\Vert Du\Vert(U).
\end{align*}
By the weak compactness of BV functions, see~\cite[Theorem 3.7]{M}, a subsequence that we still 
denote by $h_i$ converges in $L^1_{\loc}(X)$ to a function $h\in\BV(X)$ for which 
we clearly have $h=0$ in $X\setminus U$.
Now let $w:=u+h\in\BV(U)$. Since
$\Vert h_i\Vert_{L^1(U)}=\Vert v_i-u\Vert_{L^1(U_i)}\le \kappa$ for each $i\in\N$, it follows that
$\Vert w-u\Vert_{L^1(U)}=\Vert h\Vert_{L^1(U)}\le \kappa$.

To prove that $\widetilde{w}\in \liploc(U)\cap N^{1,1}(U)$, pick a Whitney type covering $\{B_j\}_{j\in\N}$ of $U$ at scale $R$, and fix a ball $B_j=B(x_j,r_j)$. For large enough $i_0\in\N$,
\begin{align*}
B_j &= B(x_j,\min\{\dist(x_j,X\setminus U)/40\lambda,R\})\\
 &\subset B(x_j,\min\{\dist(x_j,X\setminus U_{i_0})/20\lambda,R\})=:B.
\end{align*}
By~\eqref{eq:lipschitz constant of discrete convolution} we 
know that in the ball $B$, each $v_i$, $i\ge i_0$, has Lipschitz constant at most
\begin{align*}
&C\frac{\Vert Du\Vert(B(x_j,\min\{\dist(x_j,X\setminus U_{i})/4,5\lambda R\}))}
{\mu(B(x_j,\min\{\dist(x_j,X\setminus U_{i})/20\lambda,R\}))}\\
&\qquad   \le C\frac{\Vert Du\Vert(B(x_j,\min\{\dist(x_j,X\setminus U)/4,5\lambda R\}))}{\mu(B(x_j,\min\{\dist(x_j,X\setminus U_{i_0})/20\lambda, R\}))} 
\le C\frac{\Vert Du\Vert(10\lambda B_j)}{\mu(B_j)}.
\end{align*}
In the last inequality we used the fact that $B_j\subset B$, so that $\mu(B_j)\le \mu(B)$. Now the $L^1$-limit $\widetilde{w}$ of the sequence of functions
$v_i$ must be Lipschitz in $B$, and thus in $B_j$, with the same constant, so that it is locally Lipschitz. Since a local Lipschitz constant is always an upper gradient, see e.g. \cite[Proposition 1.11]{Che}, we have also that
\[
g:=C\sum_{j\in\N}\chi_{B_j}\frac{\Vert Du\Vert(10\lambda B_j)}{\mu(B_j)}
\]
is an upper gradient of $\widetilde{w}$ in $U$. By the bounded overlap of the dilated Whitney balls $10\lambda B_j$, $\Vert g\Vert_{L^1(U)}\le C\Vert Du\Vert(U)$.

Now choose $x\in\partial U$ and $r>0$.
Since $h_i\to h$ in $L^1_{\loc}(X)$ and thus in $L^1(B(x,r))$, we have
\begin{align*} 
\int_{B(x,r)\cap U}|h|\,d\mu=\lim_{i\to\infty}\int_{B(x,r)\cap U_i}|h_i|\,d\mu
&= \lim_{i\to\infty}\int_{B(x,r)\cap U_i}|v_i-u|\,d\mu\\
&\le Cr\Vert Du\Vert (B(x,2 r)\cap U),
\end{align*}
where the last inequality follows from~\eqref{eq:discrete convolution boundary value calculation}.
Then by Lemma~\ref{lem:measure in domain}, we have that
\begin{equation}\label{eq:zero boundary value for mollification}
\lim_{r\to 0^+}\frac{1}{\mu(B(x,r))}\int_{B(x,r)}|h|\,d\mu=\lim_{r\to 0^+}\frac{1}{\mu(B(x,r))}\int_{B(x,r)\cap U}|h|\,d\mu=0
\end{equation}
for $\mathcal H$-a.e. $x\in\partial U$. For such $x$ and all $t\neq 0$, we 
conclude that $x\notin \partial^*\{h>t\}$. By the coarea formula~\eqref{eq:coarea}
and~\eqref{eq:def of theta}, this implies
\[
\Vert Dh\Vert(\partial U)
=\int_{\R}P(\{h>t\},\partial U)\,dt
\le C\int_{\R}\mathcal H(\partial^*\{h>t\}\cap\partial U)\,dt=0.
\]
We conclude that $\Vert Dh\Vert(X\setminus U)=0$.
\end{proof}

In this paper, we will only need the following corollary of Theorem~\ref{thm:mollifying in an open set}.
\begin{corollary}\label{cor:mollifying in an open set}
Let $U\subset\Omega\subset X$ be open sets, $u\in\BV(\Omega)$, and $\kappa>0$. 
Then there exists a function $w\in\BV(\Omega)$ with $w=u$ in $\Omega\setminus U$ such that
$\Vert w-u\Vert_{L^1(U)}\le\kappa$, $\widetilde{w}|_{U}\in N^{1,1}(U)\cap \liploc(U)$ with an
upper gradient $g$ satisfying $\Vert g\Vert_{L^1(U)}\le C\Vert Du\Vert(U)$, and
\begin{equation}\label{eq:variation measure for mollified function}
\Vert D(w-u)\Vert(\Omega\setminus U)=0.
\end{equation}
\end{corollary}

\begin{proof}
Let $w:=u+h$, where $h\in\BV(X)$ is given in Theorem~\ref{thm:mollifying in an open set}. Then $w\in\BV(\Omega)$, and the required properties of $w$ were shown in the theorem.
\end{proof}

We will also need the following consequence of Theorem~\ref{thm:mollifying in an open set}.
The proof will be similar to one given in~\cite{KKST}. 

\begin{proposition}\label{prop:uniform convergence and continuity}
Let $\Omega\subset X$ be an open set, let $u\in\BV(\Omega)$,
and let $H\subset\Omega$ be a closed set such that $\widetilde{u}|_H$ is continuous and
\begin{equation}\label{eq:uniform Lebesgue point convergence}
\vint{B(x,r)}|u-\widetilde{u}(x)|\,d\mu\to 0\qquad \textrm{as }r\to 0
\end{equation}
locally uniformly in the set $H$. Let $w$ be the function given
by Corollary~\ref{cor:mollifying in an open set}
with $U=\Omega\setminus H$ and any $\kappa>0$.
Then $\widetilde{w}$ is continuous in $\Omega$, and $\widetilde{w}(x)=\widetilde{u}(x)$ for all $x\in H$.
\end{proposition}

\begin{proof}
Observe that $\widetilde{w}$ is continuous
in $U=\Omega\setminus H$ by Corollary~\ref{cor:mollifying in an open set}.

Let $R$ be the scale used in the construction of the Whitney type coverings of the sets $U_i$ in
Theorem~\ref{thm:mollifying in an open set}, corresponding to the given value of $\kappa$.
Fix $x\in H$. If $x$ is in the interior of $H$, then $\widetilde{w}$, which agrees with $\widetilde{u}$ in the interior of $H$, 
is continuous at $x$. Now suppose that $x$ is not in the interior of $H$.
Let $\delta\in (0,R)$ such that $B(x,3\delta)\subset \Omega$.
Consider a sequence $y_k\in B(x,\delta)\setminus H$ that converges to $x$.
We note that for every $y_k$ there exists $x_k\in H$ 
for which $d(y_k,x_k)=\dist(y_k,H)$. Since $d(y_k,x_k)\le d(y_k,x)\to 0$ as $k\to \infty$,  it follows that 
$d(x_k,x)\to 0$. The latter, together with the assumption that $\widetilde{u}|_H$ is continuous, 
implies that $\widetilde{u}(x_k)\to \widetilde{u}(x)$. So we only need to show that 
$|\widetilde{w}(y_k)-\widetilde{u}(x_k)|\to 0$ as $k\to\infty$.

Fix $k\in\N$. For large enough $i\in\N$, the sets $U_i$ (defined in the proof of 
Theorem~\ref{thm:mollifying in an open set}) 
satisfy
\begin{equation}\label{eq:property of Ui for large i}
U_i\supset \{y\in B(x,2\delta) :\,  \dist(y,H)>\dist(y_k,H)/2\}\ni y_k. 
\end{equation}
Fixing such $i$, by the 
properties of the Whitney type covering $\{B^i_j\}_{j\in \N}=\{B(x^i_j,r^i_j)\}_{j\in \N}$ of  $U_i$,
for any $2B^i_j \ni y_k$ we have 
\begin{equation}\label{eq:Madayan1}
40\lambda r_j^i\le \dist(x_j^i, H)\le 2r_j^i+\dist(y_k,H)<2r_j^i+\delta,
\end{equation}
and it follows that $r_j^i<R$. Therefore
\[
r_j^i=\min\left\{\frac{\dist(x_j^i, X\setminus U_i)}{40\lambda},R\right\}
=\frac{\dist(x_j^i, X\setminus U_i)}{40\lambda}\ge \frac{\dist(y_k,X\setminus U_i)-2r_j^i}{40\lambda},
\]
from which we see that 
\[
r_j^i\ge \frac{\dist(y_k,X\setminus U_i)}{50\lambda}\overset{\eqref{eq:property of Ui for large i}}{\ge} \frac{\dist(y_k,H)}{100\lambda} 
=\frac{d(y_k,x_k)}{100\lambda}.
\]
Thus by the doubling property of $\mu$, $C\mu(2B_j^i)\ge \mu(B(x_k,2d(x_k,y_k)))$ for 
$C=C(C_d,\lambda)$. Furthermore, by the first two inequalities of~\eqref{eq:Madayan1},
\[
 r_j^i\le \frac{d(y_k,x_k)}{38\lambda},
\]
so that $2B_j^i\subset B(x_k,2d(x_k,y_k))$.  Recall that $w$ was defined in the proof of Theorem~\ref{thm:mollifying in an open set} as the limit of the discrete convolutions $v_i$
of $u$ in $U_i$. Noting that $k$ and $i$ are fixed 
and that the summations below are over indices $j$, we have
\begin{align*}
|v_i(y_k)-\widetilde{u}(x_k)|
&=\Bigg|\sum_{y_k\in 2B_j^i}\phi_j^i(y_k)(u_{B_j^i}-\widetilde{u}(x_k))\Bigg|\\
&\le C\sum_{y_k\in 2B_j^i}\vint{2B_j^i}|u-\widetilde{u}(x_k)|\,d\mu\\
&\le C\sum_{y_k\in 2B_j^i}\vint{B(x_k,2d(x_k,y_k))}|u-\widetilde{u}(x_k)|\,d\mu\\
&\le C C_0\vint{B(x_k,2d(x_k,y_k))}|u-\widetilde{u}(x_k)|\,d\mu,
\end{align*}
where $C_0$ was the overlap constant of the Whitney balls. Letting $i\to\infty$, we get
\[
|\widetilde{w}(y_k)-\widetilde{u}(x_k)|\le C\vint{B(x_k,2d(x_k,y_k))}|u-\widetilde{u}(x_k)|\,d\mu,
\]
which converges to 0 as
$k\to\infty$, because the convergence in~\eqref{eq:uniform Lebesgue point convergence} was locally uniform. In total, 
$\widetilde{w}(y_k)\to \widetilde{u}(x)$ and then $\widetilde{w}(x)=\widetilde{u}(x)$ for every $x\in H$, so we have the 
desired conclusion.
\end{proof}

\section{Proof of Theorem~\ref{thm:main result}: outside the jump set}\label{sec:proof outside jump set}

In this section we use the tools developed in the previous section to prove one part of the main theorem of this
paper, Theorem~\ref{thm:main result}.
As a by-product, we obtain some approximation results for $\BV$ functions. First, we highlight some properties
of the $1$-capacity $\capa_1$ relevant to this paper --- recall the definition from~\eqref{eq:cap-def}.

\begin{remark}\label{rmk:capacities etc}
From \cite[Lemma 3.4]{KKST3} it follows that $\capa_1(A)\le 2C_d\mathcal H_1(A)$ for any
$A\subset X$. 
On the other hand, by combining~\cite[Theorem~4.3]{HaKi} and the proof of~\cite[Theorem 5.1]{HaKi}, 
we know that $\mathcal H_{\eps}(A)\le C(C_d,C_P,\lambda,\eps)\capa_1(A)$ for 
any $A\subset X$ and $\eps>0$. Thus we could also control the size of the "exceptional set" $G$ in Theorem~\ref{thm:main result} 
and elsewhere by its $\mathcal H_{\eps}$-measure, for arbitrarily small $\eps>0$.
Finally, we note that $\capa_1$ is an outer capacity, meaning that
\[
\capa_1(A)=\inf\{\capa_1(U):\,U\supset A\textrm{ is open}\}
\]
for any $A\subset X$,
see e.g. \cite[Theorem 5.31]{BB}. Thus in Theorem~\ref{thm:main result} and elsewhere we can always make 
the set $G$ open, even if its construction does not automatically make it as such.
\end{remark}

A version of the following lemma was previously known  
for Newton-Sobolev functions (see e.g. \cite{KKST3}).

\begin{lemma}\label{lem:from BV to uniform convergence}
Let $u_i,u\in\BV(X)$ with $u_i\to u$ in $\BV(X)$. Let $\eps>0$. Then there exists 
$F\subset X$ with $\capa_1 (F)<\eps$ such that, by picking a subsequence if necessary, 
$u_i^{\wedge}\to u^{\wedge}$ and $u_i^{\vee}\to u^{\vee}$ {\rm(}and thus also 
$\widetilde{u_i}\to \widetilde{u}${\rm)} uniformly in $X\setminus F$.
\end{lemma}

\begin{proof}
For $v\in\BV(X)$, by Lemma~\ref{lem:capacity level set maximal function} we know that
\[
\capa_1 (\{x\in X:\,\mathcal M_1 v(x)>t\})\le \frac{C_1}{t}\Vert v\Vert_{\BV(X)}
\]
for any $t>0$,
where $C_1$ is the constant from the lemma, corresponding to the choice $R=1$. By the coarea formula~\eqref{eq:coarea}, there exists a countable dense set 
$T\subset \R$ such that for every $s\in T$, $P(\{v>s\},X)<\infty$. 
Recall the definition of $\Sigma_\gamma E$ for sets $E\subset X$ from~\eqref{eq:density of E}. 
We set
\[
N:=\bigcup_{s\in T}\partial^*\{v>s\}\setminus\Sigma_{\gamma}\{v>s\}.
\]
By~\eqref{eq:density of E} we know that $\mathcal H(N)=0$. For $x\in X\setminus N$, 
if $t>0$ and $t<v^{\vee}(x)$, by the definition of the upper approximate limit we have that
\[
\limsup_{r\to 0^+}\frac{\mu(B(x,r)\cap \{v>t\})}{\mu(B(x,r))}>0.
\]
Then, since $x\in X\setminus N$, for any $s<v^{\vee}(x)$ with $s\in T$ we have
\[
\liminf_{r\to 0^+}\frac{\mu(B(x,r)\cap \{v>s\})}{\mu(B(x,r))}\ge\gamma.
\]
Thus for any $v\in \BV(X)$ we have that
$\mathcal M_1 v(x)\ge \gamma v^{\vee}(x)$ for any $x\in X\setminus N$, and so
\begin{equation}\label{eq:capacity estimate for superlevel sets}
\begin{split}
\capa_1(\{x\in X:\, v^{\vee}(x)>t\})
&=\capa_1(\{x\in X\setminus N:\, v^{\vee}(x)>t\})\\
&\le \capa_1(\{x\in X\setminus N:\, \mathcal M_1 v(x)>\gamma t\})\\
&\le \frac{C_1}{\gamma t}\Vert v\Vert_{\BV(X)}
\end{split}
\end{equation}
for any $t>0$. 

Now let $u_i, u$ be as in the statement of the lemma.
By picking a subsequence if necessary, we can assume that for each $i\in\N$,
$\Vert u_i-u\Vert_{\BV(X)}\le 2^{-2i}\gamma/C_1$. It is easy to check that we can write
\[
\{|u_i^{\vee}-u^{\vee}|>2^{-i}\}\subset \{|u_i-u|^{\vee}>2^{-i}\}
\cup \{|u^{\vee}|=\infty\}\cup \{|u_i^{\vee}|=\infty\}=:F_i.
\]
By \cite[Lemma 3.2]{KKST2} we know that $\mathcal H(\{|u^{\vee}|=\infty\}\cup \{|u_i^{\vee}|=\infty\})=0$, and 
then by~\eqref{eq:capacity estimate for superlevel sets}, $\capa_1(F_i)\le 2^{-i}$ for each $i\in\N$, so that for large 
enough $k\in\N$ we have
\[
\capa_1\left(\bigcup_{i=k}^{\infty}F_i\right)< \eps.
\]
Note that $u_i^{\vee}\to u^{\vee}$ uniformly in $X\setminus \bigcup_{i=k}^{\infty}F_i$. Similarly we get 
$u_i^{\wedge}\to u^{\wedge}$ uniformly in $X\setminus\bigcup_{i=k}^{\infty}F_i$.
\end{proof}

Recall that the jump set $S_u$ of a BV function $u$ is defined as the set where $u^{\wedge}<u^{\vee}$.

\begin{lemma}\label{lem:density of continuous functions}
Let $u\in\BV(X)$ with $\mathcal H(S_u)=0$. Then there exists a sequence $w_i\in \BV(X)\cap C(X)$ with $w_i\to u$ in $\BV(X)$.
\end{lemma}

\begin{proof}
By~\cite[Theorem 3.5]{KKST2} we know that
\begin{equation}\label{eq:Lebesgue points in Cantor part}
\lim_{r\to 0^+}\,\vint{B(x,r)}|u-\widetilde{u}(x)|\,d\mu=0
\end{equation}
for $\mathcal H$-a.e. $x\in X$, in particular for $\Vert Du\Vert$-a.e.~$x\in X$, as
by~\eqref{eq:def of theta} and the coarea formula~\eqref{eq:coarea},
$\Vert Du\Vert$ is always absolutely continuous with respect to $\mathcal H$.

Note that $\widetilde{u}$ is a Borel function, and hence is measurable with respect to the Radon measure $\Vert Du\Vert$.
By Lusin's theorem and Egorov's theorem, 
we can pick compact sets $H_i\subset X$ with 
$\Vert Du\Vert(X\setminus H_i)<1/i$, $i\in\N$, such that $\widetilde{u}|_{H_i}$ is continuous and the 
convergence in~\eqref{eq:Lebesgue points in Cantor part}
as $r\to 0$ is uniform in $H_i$.
For each $i\in\N$, apply Corollary~\ref{cor:mollifying in an open set} with 
$U=X\setminus H_i$ and $\kappa=1/i$, to obtain a function $w_i\in\BV(X)$
with $w_i=u$ in $H_i$.
We have $w_i\to u$ in $L^1(X)$, and
\begin{align*}
\Vert D(w_i-u)\Vert(X)
&=\Vert D(w_i-u)\Vert(X\setminus H_i)\\
&\le \Vert Dw_i\Vert(X\setminus H_i)+\Vert Du\Vert(X\setminus H_i)\\
&\le C\Vert Du\Vert(X\setminus H_i)+\Vert Du\Vert(X\setminus H_i)\\
&\le C/i
\end{align*}
for each $i\in\N$, so that in fact $w_i\to u$ in $\BV(X)$. By Proposition~\ref{prop:uniform convergence 
and continuity}, each $\widetilde{w_i}$ is continuous in $X$.
\end{proof}

\begin{remark}\label{rem:JumpSBV}
If $u\in\BV(X)$ and we have a sequence of continuous functions $u_i\to u$ in $\BV(X)$, 
then  $\Vert Du_i\Vert(S_u)=0$ by the facts that $S_u$ is $\sigma$-finite with respect to $\mathcal H$ (by e.g. the 
decomposition~\eqref{eq:decomposition}) and $\Vert Du_i\Vert^j(X)=0$ for all $i\in\N$, see~\cite[Theorem 5.3]{AMP}. Thus
\[
\Vert Du\Vert(S_u)=\Vert D(u-u_i)\Vert(S_u)\le \Vert D(u-u_i)\Vert(X)\to 0
\]
as $i\to\infty$, so that $\mathcal H(S_u)=0$. Hence 
the subspace $\{u\in\BV(X):\,\mathcal H(S_u)=0\}$ is the closure of $\BV(X)\cap C(X)$ in $\BV(X)$.
\end{remark}

Now we turn to our first 
quasicontinuity result.

\begin{proposition}\label{prop:quasicontinuity without jump part}
Let $u\in\BV(X)$ with $\mathcal H(S_u)=0$, and let $\eps>0$. Then there exists 
$G\subset X$ 
with $\capa_1(G)<\eps$ such that $\widetilde{u}|_{X\setminus G}$ is continuous.
\end{proposition}

\begin{proof}
By Lemma~\ref{lem:density of continuous functions}, we can pick a sequence $w_i\in \BV(X)\cap C(X)$ with $w_i\to u$ in $\BV(X)$.
By Lemma~\ref{lem:from BV to uniform convergence} 
there exists, by passing to a subsequence if necessary, 
a set $F\subset X$ with $\capa_1(F)<\eps$ such that $w_i\to \widetilde{u}$ 
uniformly in $X\setminus F$. Thus $\widetilde{u}|_{X\setminus F}$ is continuous.
\end{proof}

\begin{corollary}\label{cor:quasicontinuity without jump part for open set}
Let $\Omega\subset X$ be an open set and $u\in\BV(\Omega)$ with $\mathcal H(S_u)=0$. 
Let $\eps>0$. Then there exists $G\subset \Omega$ with $\capa_1(G)<\eps$ 
such that $\widetilde{u}|_{\Omega\setminus G}$ is continuous.
\end{corollary}

\begin{proof}
We denote
\[
\Omega_\delta:=\{x\in\Omega:\, \dist(x,X\setminus\Omega)>\delta\},\qquad \delta>0.
\]
For $\delta>0$, take $\eta_{\delta}\in \Lip(X)$ with $0\le \eta_{\delta}\le 1$, $ \eta_{\delta}=1$ in $\Omega_{\delta}$, 
and $\eta_{\delta}=0$ outside $\Omega_{\delta/2}$. Then clearly $u\eta_{\delta}\in \BV(X)$. 
By the previous proposition, for each $i\in\N$ there exists $G_i\subset \Omega$ with 
$\capa_1(G_i)< 2^{-i}\eps$ such that $\widetilde{u\eta_{1/i}}|_{X\setminus G_i}$ is continuous, and 
clearly $\widetilde{u\eta_{1/i}}\equiv \widetilde{u}\eta_{1/i}=\widetilde{u}$ in $\Omega_{1/i}$. Define 
$G:=\bigcup_{i\in\N} G_i$. Then for each $i\in\N$, $\widetilde{u}|_{\Omega_{1/i}\setminus G}$ is continuous, 
whence $\widetilde{u}|_{\Omega\setminus G}$ is continuous, and $\capa_1(G)<\eps$.
\end{proof}

Now we can prove Theorem~\ref{thm:main result} for points outside the jump set of a BV function.

\begin{proposition}\label{prop:quasicontinuity for BV}
Let $u\in\BV(X)$ and let $\eps>0$. Then there exists $G\subset X$ with
$\capa_1(G)<\eps$ such that
whenever $y_k\to x$ with $y_k\in X\setminus G$ and $x\in X\setminus (G\cup S_u)$, then 
$u^{\wedge}(y_k)\to \widetilde{u}(x)$ and $u^{\vee}(y_k)\to \widetilde{u}(x)$.
\end{proposition}

Note that the conclusion of the proposition is stronger than saying that $\widetilde{u}(y_k)\to\widetilde{u}(x)$.

\begin{proof}
Since $\Vert Du\Vert$ is a Radon measure and $S_u$ is a Borel set, we can find compact sets $H_i\subset S_u$ with 
$\Vert Du\Vert(S_u\setminus H_i)<1/i$ for each $i\in\N$.
For each $i\in\N$, take an open set
$U_i\subset X$ with $U_i\supset S_u\setminus H_i$ and $\Vert Du\Vert(U_i)<1/i$, and apply
Corollary~\ref{cor:mollifying in an open set} with $U= U_i$ and $\kappa=1/i$ to obtain a function 
$w_i\in\BV(X)$ with $w_i=u$ in $X\setminus U_i$.
We have $w_i\to u$ in $L^1(X)$, and
by~\eqref{eq:variation measure for mollified function},
\begin{align*}
\Vert D(w_i-u)\Vert(X)
&=\Vert D(w_i-u)\Vert(U_i)\\
&\le \Vert Dw_i\Vert(U_i)+\Vert Du\Vert(U_i)\\
&\le C\Vert Du\Vert(U_i) \le C/i
\end{align*}
for each $i\in\N$, so in fact $w_i\to u$ in $\BV(X)$. By Corollary~\ref{cor:mollifying in an open set},
for each $i\in\N$, $\widetilde{w_i}$ is continuous in $U_i$ and hence has no jump part there; therefore 
by~\eqref{eq:variation measure for mollified function},
\begin{align*}
\Vert Dw_i\Vert^j(X\setminus H_i)
&= \Vert Dw_i\Vert^j(U_i\setminus H_i)+\Vert Du\Vert^j(X \setminus (H_i\cup U_i))\\
&\le \Vert Dw_i\Vert^j(U_i)+\Vert Du\Vert^j(X \setminus S_u)=0,
\end{align*}
so the jump set of $w_i$ satisfies $\mathcal H(S_{w_i}\setminus H_i)=0$. Thus 
by Corollary~\ref{cor:quasicontinuity without jump part for open set} applied to the open set that is $X\setminus H_i$, 
there exists $G_i\subset X$ with $\capa_1(G_i)< 2^{-i-1} \eps$ such that 
$\widetilde{w_i}|_{X\setminus (H_i\cup G_i)}$ is continuous. Since also
$\capa_1(S_{w_i}\setminus H_i)=0$, we can assume that 
$G_i\supset S_{w_i}\setminus H_i$, 
so that $\widetilde{w_i}=w_i^{\wedge}=w_i^{\vee}$ in 
$X\setminus (H_i\cup G_i)$. Let ${G}_\infty:=\bigcup_{i\in\N}G_i$. Since $w_i\to u$ in 
$\BV(X)$, by Lemma~\ref{lem:from BV to uniform convergence} and by picking a subsequence 
if necessary, there exists $F\subset X$ with $\capa_1(F)<\eps/2$ such that $w_i^{\wedge}\to u^{\wedge}$ 
and $w_i^{\vee}\to u^{\vee}$ uniformly in $X\setminus F$.
For $G:=F\cup {G}_\infty$, clearly $\capa_1(G)<\eps$.

Finally, let $y_k\to x$ with $y_k\in X\setminus G$ and $x\in X\setminus (S_u\cup G)$. 
Note that since each $H_i\subset S_u$ is compact, for each $i\in\N$ we necessarily have 
$y_k\in X\setminus H_i$ for large enough $k$, so for these indices, 
$\widetilde{w_i}(y_k)=w_i^{\wedge}(y_k)=w_i^{\vee}(y_k)$. For some sequence of nonnegative 
numbers $\alpha_i\to 0$, we have
\[
|w_i^{\wedge}(z)-u^{\wedge}(z)|\le \alpha_i\qquad
\textrm{and}\qquad |w_i^{\vee}(z)-u^{\vee}(z)|\le \alpha_i
\]
for all $z\in X\setminus F$, $i\in\N$. Thus
\begin{align*}
&\limsup_{k\to\infty}|u^{\wedge}(y_k)-\widetilde{u}(x)|\\
&\ \ \ \  \le \limsup_{k\to\infty} \,(|u^{\wedge}(y_k)-w_i^{\wedge}(y_k)|+|w_i^{\wedge}(y_k)-\widetilde{w_i}(x)|+
|\widetilde{w_i}(x)-\widetilde{u}(x)|)\\
&\ \ \ \ \le\limsup_{k\to\infty} \,(\alpha_i+|w_i^{\wedge}(y_k)-\widetilde{w_i}(x)|+\alpha_i)\\
&\ \ \ \ =\limsup_{k\to\infty} |\widetilde{w_i}(y_k)-\widetilde{w_i}(x)|+2\alpha_i\\
&\ \ \ \ =2\alpha_i
\end{align*}
by the continuity of $\widetilde{w_i}|_{X\setminus (H_i\cup G_i)}$. Letting $i\to\infty$ completes the proof
for $u^\wedge$. For $u^{\vee}$, the proof is the similar. This
completes the proof of the proposition.
\end{proof}

\section{Proof of Theorem~\ref{thm:main result}: within the jump set} 

In this section we complete the proof of Theorem~\ref{thm:main result}. 
First we consider a generalization of the Lebesgue differentiation theorem for the jump set of a $\BV$ function.
Recall the definition of the number $Q>0$ from~\eqref{eq:definition of Q}. We know from \cite[Theorem~4.3]{L} that for 
$u\in \BV(X)$ and $\mathcal H$-a.e. $x\in S_u$, there exist $t_1,t_2\in (u^{\wedge}(x),u^{\vee}(x))$ such that
\[
\lim_{r\to 0^+}\,\vint{B(x,r)\cap \{u< t_1\}}|u-u^{\wedge}(x)|^{Q/(Q-1)}\,d\mu=0
\]
and
\[
\lim_{r\to 0^+}\,\vint{B(x,r)\cap \{u>t_2\}}|u-u^{\vee}(x)|^{Q/(Q-1)}\,d\mu=0.
\]
We cannot in general pick $t_1,t_2$ freely from the interval $(u^{\wedge}(x),u^{\vee}(x))$, as we can in the Euclidean setting, 
as demonstrated by the following example.

\begin{example}\label{ex:one dimensional space}
Consider the one-dimensional space
\[
X:=\{(x_1,x_2) \in \R^2:\,x_1=0\text{ or }x_2=0\}
\]
consisting of the two coordinate axes.
Equip this space with the Euclidean metric inherited from $\R^2$, and the 1-dimensional Hausdorff measure. 
This measure is doubling and supports a $(1,1)$-Poincar\'e inequality.
Let 
\[
u:=\chi_{\{x_1>0\}}+2\chi_{\{x_2>0\}}+3\chi_{\{x_1<0\}}+4\chi_{\{x_2<0\}}.
\]
For brevity, denote the origin $(0,0)$ by $0$. Now $S_u=\{0\}$ with $\mathcal H(\{0\})=2$, and
$(u^{\wedge}(0),u^{\vee}(0))=(1,4)$. However, 
we cannot choose $t_1$ to be larger than 2, nor $t_2$ to be smaller than 3. This demonstrates that in 
a metric space, a BV function can, in a measure theoretic sense, take more than 2 values all along its jump set $S_u$.

Higher-dimensional example spaces can be obtained by simply taking Cartesian products of $X$ with e.g. $[0,1]$.
\end{example}

\begin{example}
Closely related to this issue are the~\emph{locality conditions} discussed in~\cite{AMP} and \cite{L}.
We say that $X$ supports the~\emph{strong locality condition} if for every pair of 
sets $E_1\subset E_2\subset X$ of finite perimeter, we have
\[
\lim_{r\to 0^+}\frac{\mu(B(x,r)\cap(E_2\setminus E_1))}{\mu(B(x,r))}=0
\]
for $\mathcal H$-a.e. $x\in\partial^*E_1\cap\partial^*E_2$.
Following \cite{AMP}, we also say that $X$ supports the \emph{locality condition} if for every pair of sets $E_1\subset E_2\subset X$ of 
finite perimeter, we have $\theta_{E_1}(x)=\theta_{E_2}(x)$ for $\mathcal H$-a.e. 
$x\in\partial^*E_1\cap\partial^*E_2$ (the function $\theta_{E}$ was defined in~\eqref{eq:def of theta}).
In \cite[Proposition 6.2]{AMP}, the authors show that the strong locality condition implies the locality condition.
In \cite[Theorem~4.10]{L} it was shown that if the space supports the strong locality condition, then every pair
$t_1,t_2$ from the interval $(u^{\wedge}(x),u^{\vee}(x))$
satisfies the two equations from the beginning of
this section. However, either locality condition can fail
in a metric space, even one with a doubling measure supporting a Poincar\'e inequality. Consider the space from 
Example~\ref{ex:one dimensional space}. The sets
\[
E_1:=\{x_1>0\},\qquad E_2:=\{x_1>0\}\cup\{x_2>0\}
\]
are easily seen to be of finite perimeter, and $\partial^*E_1=\partial^*E_2=\{0\}$, that is, the measure theoretic boundaries 
only contain the origin. We have $\mathcal H(\{0\})=2$. The strong locality condition fails at the origin, since
\[
\lim_{r\to 0^+}\frac{\mu(B(0,r)\cap (E_2\setminus E_1))}{\mu(B(0,r))}
  =\lim_{r\to 0^+}\frac{\mu(B(0,r)\cap \{x_2>0\})}{\mu(B(0,r))}=\frac{1}{4}.
\]

In addition, we see that $P(E_1,X)=1$, since we can take approximating Lipschitz functions with support in $\{x_1>0\}$. But this 
does not work for $E_2$, and so we get $P(E_2,X)=2$. On the other hand, obviously 
$\mathcal H(\partial^*E_1)=\mathcal H(\partial^*E_2)$, because both sets consist of the same point. Thus 
$\theta_{E_1}(0)=1/2$ but $\theta_{E_2}(0)=1$, and the locality condition fails as well.
\end{example}

Recall the definition of $\gamma>0$ from~\eqref{eq:density of E}, the definition $n=\lfloor 1/\gamma\rfloor$, 
and the definition of the functions $u^l$ (defined also below)
for $u\in\BV(X)$ from~\eqref{eq:definition of the n limits}.
Denote by $n(x)$ the number of distinct values $u^l(x)$, $l\in\{1,\ldots,n\}$.
Also, for 
$u\in\BV(X)$, $x\in X$, and $\delta>0$, we denote
\begin{equation}\label{eq:definition of Als}
\begin{split}
& A_{l}^{\delta}(x):=[u^{l-1}(x)+\delta, u^{l+1}(x)-\delta],\qquad l=2,\ldots,n(x)-1,\\
& A_{1}^{\delta}(x):=(-\infty, u^{2}(x)-\delta],\qquad A_{n(x)}^{\delta}(x):=[u^{n(x)-1}(x)+\delta,\infty).
\end{split}
\end{equation}

\begin{theorem}\label{thm:behavior in jump set}
Let $u\in \BV(X)$.
Then for $\mathcal H$-a.e. $x\in S_u$, the following two properties hold:
$-\infty<u^1(x)<\ldots< u^{n(x)}(x)<\infty$, and
\begin{equation}\label{eq:behavior in jump set 3}
\lim_{r\to 0^+}\,\vint{B(x,r)\cap \{u\in A_l^{\delta}(x)\}}|u-u^{l}(x)|^{Q/(Q-1)}\,d\mu=0
\end{equation}
for each $l=1,\ldots,n(x)$, and every
\[
0<\delta\le\min\{u^{j+1}(x)-u^{j}(x),\ j=1,\ldots,n(x)-1\}/2.
\]
For $l=2,\ldots,n(x)-1$, we can in fact replace $Q/(Q-1)$ with any $q>0$.
\end{theorem}

\begin{proof}
This is a generalization of results in~\cite{L}.  Denote, for brevity, the super-level sets of $u$ by $E_t:=\{u>t\}$, $t\in\R$.
By the coarea formula~\eqref{eq:coarea}, there is a countable dense set $T\subset \R$ such that for every $t\in T$, 
the set $E_t$ is of finite perimeter. Let
\begin{equation*}
N:=\bigcup_{t\in T} \partial^*E_t\setminus \Sigma_{\gamma}E_t
\end{equation*}
and
\begin{equation*}
\widetilde{N}:=\bigcup_{s,t\in T:\,s<t} \partial^*(E_s\setminus E_t)\setminus \Sigma_{\gamma}(E_s\setminus E_t).
\end{equation*}
Recalling~\eqref{eq:density of E}, and since the sets $E_s\setminus E_t$, $s,t\in T$, are also of finite perimeter 
by~\cite[Proposition 4.7]{M}, we have $\mathcal H(N\cup\widetilde{N})=0$.

Fix $x\in S_u\setminus (N\cup \widetilde{N})$. By discarding another $\mathcal H$-negligible set, we can assume that 
$u^{\wedge}(x),u^{\vee}(x)$ are finite, see~\cite[Lemma 3.2]{KKST2}.  Set $u^1(x)=u^\wedge(x)$, and
define inductively for
$l=2,\ldots, n-1=\lfloor 1/\gamma\rfloor-1$
\begin{equation*}
u^{l}(x):=\sup\left\{t\in\overline{\R}:\,\lim_{r\to 0^+}\frac{\mu(B(x,r)\cap \{u^{l-1}(x)+\eps<u<t\})}{\mu(B(x,r))}=0\ \ \forall\, \eps>0\right\}
\end{equation*}
provided $u^{l-1}(x)<u^\vee(x)$, and otherwise set $u^l(x):=u^\vee(x)$. We also set $u^n(x):=u^\vee(x)$.
Fix $l$ and suppose that $u^l(x)<u^\vee(x)$. We can find $t_i\in T$ with $u^l(x)<t_i<u^\vee(x)$ for
each $i\in\N$ such that $t_i\searrow u^l(x)$ as $i\to\infty$. Then 
whenever $\partial^*\{t_{i+1}\le u<t_i\}$ has density $1$ at $x$ or $x\in\partial^*\{t_{i+1}\le u<t_i\}$, we must have
\[
\liminf_{r\to0^+}\frac{\mu(\{t_{i+1}\le u<t_i\}\cap B(x,r))}{\mu(B(x,r))}\ge \gamma.
\]
By the choice of $n$, this can happen only for at most $n$ number of indices $i$ (because the sets
$\{t_{i+1}\le u<t_i\}$ are pairwise disjoint). It follows
that for sufficiently large $i$, the sets $\{t_{i+1}\le u< t_i\}$ have density $0$ at $x$.
Thus if $u^l(x)<u^\vee(x)$, necessarily $u^{l+1}(x)>u^l(x)$. Thus $u^{n(x)}(x)=u^n(x)$.

By the definition of the functions $u^l$, we have
\[
\limsup_{r\to 0^+}\frac{\mu(B(x,r)\cap \{u^{l}(x)-\eps<u<u^{l}(x)+\eps\})}{\mu(B(x,r))}>0
\]
for every $\eps>0$ and all $l=1,2,\ldots,n$. 
Since $x\notin \widetilde{N}$, we have in fact
\begin{equation}\label{eq:density of difference of levels sets}
\liminf_{r\to 0^+}\frac{\mu(B(x,r)\cap \{u^{l}(x)-\eps<u<u^{l}(x)+\eps\})}
{\mu(B(x,r))}\ge \gamma.
\end{equation}
Now, if for
\[
\alpha:=\sup\left\{t\in\overline{\R}:\,\lim_{r\to 0^+}\frac{\mu(B(x,r)\cap \{u^{n(x)-1}(x)+\eps<u<t\})}{\mu(B(x,r))}=0\ \ \forall\, \eps>0\right\}
\]
we have $\alpha<u^{n(x)}(x)=u^n(x)$, then necessarily $n(x)=n$, and as above, we can conclude that for every $\eps>0$, the set $\{\alpha-\eps<u<\alpha+\eps\}$ has lower density at least $\gamma$ at $x$. Moreover, as the sets $\{\alpha-\eps<u<\alpha+\eps\}$ and $\{u^{l}(x)-\eps<u<u^{l}(x)+\eps\}$, $l=1,\ldots,n$ are all disjoint for small enough $\eps$, this contradicts the definition $n=\lfloor 1/\gamma \rfloor$. Thus $\alpha=u^{n(x)}(x)$.

For $l=1,\ldots,n(x)$, we note that by the definition of the numbers $u^{l}(x)$ and the fact that $\alpha=u^{n(x)}(x)$, the set
\[
\{u\in A_l^{\delta}(x)\}\setminus \{u^{l}(x)-\eps<u<u^{l}(x)+\eps\}
\]
has density 0 at $x$ for any $\eps>0$, and this together with~\eqref{eq:density of difference of levels sets} 
implies for any $l=2,\ldots,n(x)-1$ and $q>0$ that
\[
\lim_{r\to 0^+}\frac{1}{\mu(B(x,r))}\int_{B(x,r)\cap \{u\in A_l^{\delta}(x)\}}|u-u^{l}(x)|^{q}\,d\mu=0.
\]
By combining this with~\eqref{eq:density of difference of levels sets}, we get~\eqref{eq:behavior in jump set 3}. 
The cases $l=1$ and $l=n(x)$ require additional computations, 
since we integrate over 
sets where $u$ may be unbounded, but these cases were already covered in~\cite[Theorem 4.3]{L}. 
\end{proof}

Thus we have a rather complete measure theoretic description of the behavior of a BV function in its jump set: at 
$\mathcal H$-almost every point $x\in S_u$, the space $X$ can be partitioned into at most $\lfloor 1/\gamma\rfloor$ 
sets such that in each set, 
$u$ converges in a Lebesgue point sense to a real number in the interval 
$[u^{\wedge}(x),u^{\vee}(x)]$. Note that in Example~\ref{ex:one dimensional space}, we have $\gamma=1/4$.

\begin{proposition}\label{prop:one-sided continuity}
Let $u\in\BV(X)$ and let $\eps>0$. Then there exists $G\subset X$ with
$\capa_1 (G)<\eps$ such that if $y_k\to x$ with $y_k\in X\setminus G$ and $x\in S_u\setminus G$, then
\begin{equation}\label{eq:upper hemicontinuity in jump set}
\min_{l_2\in\{1,\ldots,n\}}|u^{l_1}(y_k)-u^{l_2}(x)|\to 0
\end{equation}
for every $l_1=1,\ldots,n$.
\end{proposition}

\begin{proof}
For $x\in X$, set
\[
\delta(x):=\min\{u^{l+1}(x)-u^l(x),\ l=1,\ldots,n(x)-1\}/2.
\]
We divide the proof into two steps.
\paragraph{Step 1.}
First assume that we have a compact set $H\subset S_u$ where $n(x)$ is constant,  
the functions $-\infty<u^1<\ldots <u^{n(x)}<\infty$ are continuous, and
\begin{equation}\label{eq:uniform Lebesgue points in jump set}
\frac{1}{\mu(B(x,r))}\int_{B(x,r)\cap \{u\in A_{l}^{\delta(x)/8}(x)\}}|u-u^l(x)|\,d\mu\to 0\quad\textrm{as }r\to 0
\end{equation}
uniformly in the set $H$ for every $l=1,\ldots,n(x)$.

We will demonstrate that there is a set $\widetilde{G}\subset X$ with $\capa_1(\widetilde{G})<\eps$ such that 
whenever $y_k\to x$ with $y_k\in X\setminus (H\cup \widetilde{G})$, $x\in H$, and 
$u^{l_1}(y_k)\in A^{\delta(x)}_{l_2}(x)$ for given $l_1\in\{1,\ldots,n\}$ and $l_2\in\{1,\ldots,n(x)\}$, then
\begin{equation}\label{eq:results in jump set in special case}
\lim_{k\to\infty}|u^{l_1}(y_k)-u^{l_2}(x)|= 0.
\end{equation}
In other words, we have continuity up to the jump set as long as we approach it from a specific "side", 
more precisely a specific level set of $u$.

For $p\in\N$, let 
\[
 A_p:=\{x\in X :\, 2^{-p-1}\le \dist(x,H)<2^{-p}\}.
\]
Since $\Vert Du\Vert(X)$ is finite and the sets $A_p$ are pairwise disjoint, we have
\[
\sum_{p\in \N}\Vert Du\Vert(A_p)<\infty.
\]
It follows that for each $j\in\N$ there exists $N_j\in\N$ such that
\[
\sum_{p\ge N_j}\Vert Du\Vert(A_p)\le 4^{-j}\eps.
\]
We can choose $j\mapsto N_j$ to be strictly increasing. We set
$a_p:=2^{-j}$ for $N_{j}< p\le N_{j+1}$, so that $a_p\to 0$ as $p\to\infty$. Now
\begin{equation}\label{eq:weighted sum}
\sum_{p\ge N_1+1}\frac{\Vert Du\Vert(A_{p-1})}{a_p}
  =\sum_{j\in\N}\sum_{p=N_j+1}^{N_{j+1}}2^{j}\Vert Du\Vert(A_{p-1})\le \sum_{j\in\N}2^{-j}\eps=\eps.
\end{equation}
Let
\[
G_p:=\left\{z\in A_p:\,\exists\, 0<r_z< 2^{-p-2}/\lambda\ \ \textrm{s.t.}\ \ \vint{B(z,r_z)}|u-u_{B(z,r_z)}|\,d\mu>a_p\right\}.
\]
Pick $p\ge 2$ and take a cover $\{B(z,\lambda r_z)\}_{z\in G_p}$ of $G_p$. By the 5-covering theorem, we can select a 
countable disjoint subcollection $\{\lambda B_j=B(z_j,\lambda r_j)\}_{j\in\N}$ such that the balls $5\lambda B_j$ cover $G_p$. For 
each $j\in\N$, we have by the Poincar\'e inequality
\[
a_p<\vint{B_j}|u-u_{B_j}|\,d\mu\le Cr_j\frac{\Vert Du\Vert
(\lambda B_j)}{\mu(B_j)}.
\]
Since all the radii necessarily satisfy $5\lambda r_j\le 1$, 
\begin{align*}
\capa_1(G_p) &\le
C\mathcal H_{1}(G_p)\le C\sum_{j\in\N}\frac{\mu(5\lambda B_j)}{5\lambda r_j}\le C\sum_{j\in\N}\frac{\mu(B_j)}{r_j} 
\le C\sum_{j\in\N}\frac{\Vert Du\Vert(\lambda B_j)}{a_p}\\
&\le \frac{C}{a_p}\Vert Du\Vert(A_{p-1}\cup A_p\cup A_{p+1}).
\end{align*}
In the last inequality we used the fact that the balls $\lambda B_j$ are disjoint.
Defining $G:=\bigcup_{p\ge N_1+1}G_p$, we have by~\eqref{eq:weighted sum}
\[
\capa_1(G)\le \sum_{p\ge N_1+1}\capa_1(G_p)\le C\sum_{p\ge N_1+1}\frac{\Vert Du\Vert(A_{p-1})}{a_p}\le C\eps.
\]

We need to prove an analog of Proposition~\ref{prop:uniform convergence and continuity}, this time not for $\widetilde{u}$ but for the functions $u^l$.
For each $m\in\N$, set $W_m:=\bigcup_{p= m}^{\infty}A_p$,
and apply Corollary~\ref{cor:mollifying in an open set} with
$U=W_m$ and $\kappa=\kappa_m\searrow 0$ to obtain a function $w^m\in\BV(X)$. By the proof of Theorem~\ref{thm:mollifying in an open set}, we can assume that the scale of the corresponding Whitney type coverings is fixed with $R=1$. Fix $m\ge N_1+1$.

Consider a sequence $y_k\to x$ with $y_k\in X\setminus (H\cup G)$ and $x\in H$, such
that for a fixed $l_2\in\{1,\ldots,n(x)\}$, $w^m(y_k)\in A_{l_2}^{\delta(x)/2}(x)$
for each $k\in\N$. For each $y_k$ let $x_k\in H$ such that $d(y_k,x_k)=\dist(y_k,H)$.
Clearly $d(y_k,x_k)\to 0$ as $k\to \infty$, and thus also $d(x_k,x)\to 0$, whence
$u^{l_2}(x_k)\to u^{l_2}(x)$. Thus we need to show that $|w^m(y_k)-u^{l_2}(x_k)|\to 0$
as $k\to\infty$.

Define $B_k:=B(y_k,\dist(y_k,H)/4\lambda)$ for each $k\in\N$, and then fix $y_k\in W_{m+2}$. According to the proof of Theorem~\ref{thm:mollifying in an open set}, $w^m=\lim_{i\to\infty}w_i$ 
for discrete convolutions
\[
w_i=\sum_{j\in\N}u_{B^i_j}\phi^i_j
\]
defined in open sets $U_i\subset W_m$, $i\in\N$, at scale $R=1$.
For large enough $i\in\N$ so that
\[
U_i\supset W_{m+1}\cap\left\{z\in X:\,\dist(z,H)\ge\dist(y_k,H)/2\right\},
\]
we have for all $2B^i_j\ni y_k$ that $B_j^i\subset B_k$ with radii comparable to $\dist(y_k,H)$. Thus
\begin{align*}
|w_i(y_k)-u_{B_k}|
&\le \sum_{j\in\N}|\phi_j^i(y_k)| |u_{B^i_j}-u_{B_k}|\\
&= \sum_{j\in\N,\, 2B_j^i\ni y_k}|\phi_j^i(y_k)| |u_{B^i_j}-u_{B_k}|\\
&\le C_0\vint{B^i_j} |u-u_{B_k}|\,d\mu\\
& \le C\vint{B_k}|u-u_{B_k}|\,d\mu.
\end{align*}
By taking the limit $i\to\infty$, we get
\begin{equation}\label{eq:discrete convolution and mean}
|w^m(y_k)-u_{B_k}|\le C\vint{B_k}|u-u_{B_k}|\,d\mu\le Ca_p,
\end{equation}
where $p\in\N$ is such that $y_k\in A_p\setminus G_p$.
As $k\to\infty$ we have $p\to\infty$, and so $a_p\to 0$. Hence $u_{B_k}\in A_{l_2}^{\delta(x)/3}(x)$ for large $k$, and
\begin{equation}\label{eq:level set fills ball}
\frac{\mu(B_k\cap \{u\notin A_{l_2}^{\delta(x)/4}(x)\})}{\mu(B_k)}\le \frac{12}{\delta(x)}\vint{B_k}|u-u_{B_k}|\,d\mu\to 0
\end{equation}
as $k\to\infty$. Therefore
\begin{equation}\label{eq:mean in level set}
\begin{split}
|u_{B_k\cap \{u\in A_{l_2}^{\delta(x)/4}(x)\}}-u_{B_k}|
&\le \vint{B_k\cap \{u\in A_{l_2}^{\delta(x)/4}(x)\}}|u-u_{B_k}|\,d\mu\\
&\le \frac{\mu(B_k)}{\mu(B_k\cap \{u\in A_{l_2}^{\delta(x)/4}(x)\})}\vint{B_k}|u-u_{B_k}|\,d\mu\\
&\to 0
\end{split}
\end{equation}
as $k\to\infty$. Now we can estimate
\begin{align*}
|w^m(y_k)- & u^{l_2}(x_k)|
\le |w^m(y_k)-u_{B_k}|\\
&+|u_{B_k}-u_{B_k\cap \{u\in A_{l_2}^{\delta(x)/4}(x)\}}|
+|u_{B_k\cap \{u\in A_{l_2}^{\delta(x)/4}(x)\}}-u^{l_2}(x_k)|.
\end{align*}
Here the first term converges to 0 as $k\to\infty$ by~\eqref{eq:discrete convolution and mean}, and the second term converges to 0 by~\eqref{eq:mean in level set}. For large enough
$k$,
by~\eqref{eq:level set fills ball} we have
$\mu(B_k\cap\{u\in A_{l_2}^{\delta(x)/4}(x)\})/\mu(B_k)\ge 1/2$, so that also
$C\mu(B_k\cap\{u\in A_{l_2}^{\delta(x)/4}(x)\})\ge \mu(B(x_k,2d(y_k,x_k)))$, 
and by the continuity of the functions
$u^l$ in $H$ we have
$|u^l(x_k)-u^l(x)|<\delta(x)/10$ for all $l=1,\ldots,n$.
Thus the third term is at most
\begin{align*}
&\vint{B_k\cap \{u\in A_{l_2}^{\delta(x)/4}(x)\}}|u-u^{l_2}(x_k)|\,d\mu\\
&\qquad\quad \le \frac{C}{\mu(B(x_k,2d(y_k,x_k)))}\int_{B(x_k,2d(y_k,x_k))\cap \{u\in A_{l_2}^{\delta(x)/4}(x)\}}|u-u^{l_2}(x_k)|\,d\mu\\
&\qquad\quad \le \frac{C}{\mu(B(x_k,2d(y_k,x_k)))}\int_{B(x_k,2d(y_k,x_k))\cap \{u\in A_{l_2}^{\delta(x_k)/8}(x_k)\}}|u-u^{l_2}(x_k)|\,d\mu,
\end{align*}
which converges to 0 by~\eqref{eq:uniform Lebesgue points in jump set}. It follows
that $|w^m(y_k)-u^{l_2}(x_k)|\to 0$ as $k\to\infty$, and since we had $u^{l_2}(x_k)\to u^{l_2}(x)$, we have $w^m(y_k)\to u^{l_2}(x)$ as $k\to\infty$.

By Corollary~\ref{cor:mollifying in an open set} we know that $w^m\to u$ in $\BV(X)$ as $m\to\infty$, and so by Lemma~\ref{lem:from BV to uniform convergence} and by picking a subsequence, if necessary, there exists
$F\subset X$ with $\capa_1(F)<\eps$ such that for some sequence $\alpha_m\searrow 0$, $|(w^m)^{\wedge}- u^{\wedge}|\le \alpha_m$ and $|(w^m)^{\vee}-u^{\vee}|\le \alpha_m$ in $X\setminus F$ for any $m\in\N$. But $(w^m)^{\wedge}=(w^m)^{\vee}=\widetilde{w^m}$ in
$W_m$, and so
\begin{equation}\label{eq:uniform convergence for w_m:s}
|\widetilde{w^m}-u^{l}|\le \alpha_m
\end{equation}
in $W_m\setminus F$ for any $l=1,\ldots,n$ and $m\in\N$. Take a sequence $y_k\to x$ with $y_k\in X\setminus (F\cup G\cup H)$,
$x\in H$,
and $u^{l_1}(y_k)\in A^{\delta(x)}_{l_2}(x)$ for given $l_1\in\{1,\ldots,n\}$ and $l_2\in\{1,\ldots,n(x)\}$. Then for sufficiently large $m\in\N$, by~\eqref{eq:uniform convergence for w_m:s} we have $\widetilde{w^m}(y_k)\in A^{\delta(x)/2}_{l_2}(x)$ for $k$ large enough such that
$y_k\in W_m$, so that
\begin{align*}
&\limsup_{k\to\infty}|u^{l_1}(y_k)-u^{l_2}(x)|\\
&\qquad\qquad \le \limsup_{k\to\infty}|u^{l_1}(y_k)-\widetilde{w^m}(y_k)|+\limsup_{k\to\infty}|\widetilde{w^m}(y_k)-u^{l_2}(x)|\\
&\qquad\qquad\le \alpha_m.
\end{align*}
Thus we have~\eqref{eq:results in jump set in special case}.

\paragraph{Step 2.}
Now we consider the general case. Partition the Borel set $S_u$ into sets $S_p$, $p=1,\ldots,n$, in which $n(x)=p$ for all 
$x\in S_p$. Since 
\[
S_p=\{u^1<\ldots<u^{p-1}=u^p<u^{\vee}\}\cup \{u^1<\ldots<u^{p-1}<u^p=u^{\vee}\},
\]
each $S_p$ is a Borel set. (Note that the set $\{u^1<\ldots<u^{p-1}=u^p<u^{\vee}\}$
is of $1$-capacity zero, by the proof of Theorem~\ref{thm:behavior in jump set}.)

For each $i\in\N$, pick compact sets $K_p^i\subset S_p$ such that for $H_i:=\bigcup_{p=1}^n K_p^i$ we have $\Vert Du\Vert(S_u\setminus H_i)<2^{-i}\eps$. By Lusin's theorem, we can assume that each $u^l$ is continuous in $H_i$, and by Theorem~\ref{thm:behavior in jump set} and Egorov's theorem we can assume that for every $x\in H_i$, $-\infty<u^1(x)<\ldots <u^{n(x)}(x)<\infty$ with
\begin{equation}\label{eq:uniform Lebesgue points in jump set 2}
\frac{1}{\mu(B(x,r))}\int_{B(x,r)\cap \{u\in A_{l}^{\delta(x)/8}(x)\}}|u-u^l(x)|\,d\mu\to 0\quad\textrm{as }r\to 0
\end{equation}
uniformly in $H_i$ for every $l=1,\ldots,n(x)$.

For any of the sets $K_p^i$, we are now in the situation described in Step 1.
Therefore for each $i\in\N$ there is a set $G_i$ with $\capa_1(G_i)<2^{-i}\eps$ such that we have the following.
Let $y_k\to x$ with $y_k\in X\setminus (K_p^i\cup G_i)$, $x\in K_p^i$,
and $u^{l_1}(y_k)\in A_{l_2}^{\delta(x)}(x)$ for some $l_1\in\{1,\ldots,n\}$ and
$l_2\in \{1,\ldots,n(x)=p\}$. Then
\begin{equation}\label{eq:using step 1}
u^{l_1}(y_k)\to u^{l_2}(x)
\end{equation}
by Step 1. Moreover,
\[
\Vert Du \Vert\left(S_u\setminus \bigcup_{i\in\N}H_i\right)=0,
\]
so that by~\eqref{eq:decomposition},
\[
\mathcal H\left(S_u\setminus \bigcup_{i\in\N}H_i\right)=0.
\]
Define
\[
G:=\bigcup_{i\in\N}G_i\cup \left(S_u\setminus \bigcup_{i\in\N}H_i\right)\cup \{|u^{\wedge}|=\infty\}\cup \{|u^{\vee}|=\infty\}.
\]
Then $\capa_1(G)<\eps$. Let $y_k\to x$ with $y_k\in X\setminus G$,
$x\in S_u\setminus G$,
and $u^{l_1}(y_k)\in A_{l_2}^{\delta(x)}(x)$ for some $l_1\in\{1,\ldots,n\}$ and $l_2\in\{1,\ldots,n(x)\}$. Note that $x\in H_i$ for some $i\in\N$. If $y_k\in H_i$, then by the continuity of the functions $u^l$ in $H_i$ we have $u^{l_1}(y_k)\to u^{l_1}(x)$, and since $u^{l_1}(y_k)\in A_{l_2}^{\delta(x)}(x)$, we necessarily have $l_1=l_2$ and thus $u^{l_1}(y_k)\to u^{l_2}(x)$. On the other hand, if $y_k\in X\setminus H_i$, then $u^{l_1}(y_k)\to u^{l_2}(x)$ by~\eqref{eq:using step 1}.

This immediately implies~\eqref{eq:upper hemicontinuity in jump set}, since we have
\[
\bigcup_{l_2=1}^{n(x)} A_{l_2}^{\delta(x)}(x)=\R
\]
at every $x\in X\setminus G$.
\end{proof}

By combining Proposition~\ref{prop:quasicontinuity for BV} and Proposition~\ref{prop:one-sided continuity} with 
the fact that $\capa_1$ is an outer capacity as noted in Remark~\ref{rmk:capacities etc},
Theorem~\ref{thm:main result} is proved.

\begin{example}
It is not true that by discarding a suitable set of small capacity $G$, we would have 
that $u^l|_{S_u\setminus G}$ is continuous for each $l$.  Consider $X=\R$ with the Euclidean distance 
and the $1$-dimensional Lebesgue measure, and set
\[
u:=\chi_{[-1,0]}+\sum_{i\in\N}2^{-i}\chi_{(2^{-i-1},2^{-i}]}.
\]
Then $u^\vee(2^{-i-1})=2^{-i}\not\to 1=u^\vee(0)$ as $i\to\infty$.
Moreover, the $1$-capacity of every point is $2$, so the only set of $1$-capacity smaller than $2$ is the
empty set.
\end{example}

\section{Application to sets of finite perimeter}\label{sec:sets of finite perimeter}

In this section we will discuss the implications of Theorem~\ref{thm:main result} for sets of finite perimeter. 
Federer's structure theorem states that a set $E\subset\R^n$ is of finite perimeter if and only if $\mathcal{H}(\partial^*E)$
is finite, see \cite[Section 4.5.11]{Fed}.
In a complete metric space $X$
with a doubling measure that supports a $(1,1)$-Poincar\'e inequality, the ``only if" direction has been shown by Ambrosio, 
see~\eqref{eq:def of theta}.
The ``if" direction was shown for a certain class of metric measure spaces in~\cite{KLS}, but remains open in general.
As part of the proof of the ``if" direction it is usually shown that the collection of lines parallel to the coordinate axes
in $\R^n$, which pass from the measure theoretic interior of $E$ to the measure theoretic exterior of $E$ but do not 
intersect $\partial^*E$, must have $1$-modulus zero, see for example the proof in~\cite[p. 222]{EvaG92}. In this section we
will prove a similar result in the metric setting,
\emph{provided} we know that $E\subset X$ is of finite perimeter. We also give a partial converse, namely that if
$E$ is a $\mu$-measurable set with $\mathcal{H}(\partial^*E)$ finite and the $1$-modulus of curves intersecting both the measure theoretic
interior of $E$ and the measure theoretic exterior of $E$ without intersecting $\partial^*E$ in between is zero, then
$E$ is of finite perimeter. (A related partial generalization was previously considered in~\cite{KoLa}.)

The measure theoretic interior $\mathcal{I}(E)$
and the measure theoretic exterior $\mathcal{E}(E)$ of a $\mu$-measurable set $E\subset X$ are defined as follows:
\[
\mathcal{I}(E):=\left\{x\in X :\, \lim_{r\to0^+}\frac{\mu(B(x,r)\cap E)}{\mu(B(x,r))}=1\right\}
\]
and
\[
\mathcal{E}(E):=\left\{x\in X :\, \lim_{r\to0^+}\frac{\mu(B(x,r)\cap E)}{\mu(B(x,r))}=0\right\}.
\]
Clearly $\partial^*E=X\setminus[\mathcal{I}(E)\cup\mathcal{E}(E)]$. Let $u=\chi_E$. Observe that $x\in \mathcal{I}(E)$ means that $u^\vee(x)=u^\wedge(x)=1$, $x\in \mathcal{E}(E)$ means that 
$u^\vee(x)=u^\wedge(x)=0$, and $x\in \partial^*E$ means that 
$u^\vee(x)=1$ and $u^\wedge(x)=0$, i.e. $x\in S_u$.

First we note that some sets of finite perimeter, such as the enlarged rationals, can exhibit 
bizarre behavior that demonstrates the necessity of excluding a set $G$ in
Theorem~\ref{thm:main result}. 

\begin{example}\label{ex:enlarged-rationals}
Let $\{q_i\}_{i\in\N}$ be an enumeration of $\mathbb{Q}\times\mathbb{Q}\subset\R^2$, and define
\[
E:=\bigcup_{i\in\N} B(q_i,2^{-i}).
\]
Clearly $\mathcal L^2(E)\le 2\pi$, and $\chi_E=\lim_j\chi_{E_j}$, where $E_j:=\bigcup_{i=1}^j B(q_i,2^{-i})$,
the limit occurring in $L^1(\R^2)$. Since 
$P(E_j,\R^2)\le 2\pi \sum_{i=1}^j2^{-i}$, we have $P(E,\R^2)<\infty$, so that also $\mathcal H(\partial^*E)<\infty$.
However, $\partial E=\R^2\setminus E$.
Thus, denoting $u:=\chi_E$, for every Lebesgue point $x\in X\setminus E$ there exists a sequence $y_k\to x$ with $y_k\in E$ such that 
\[
u^{\wedge}(y_k)=u^{\vee}(y_k)=1\not \rightarrow 0=u^{\wedge}(x)=u^{\vee}(x),
\]
so that the conclusion of Theorem~\ref{thm:main result} fails with the choice $G=\emptyset$.
On the other hand, given $\eps>0$, by choosing $G:=\bigcup_{i=k}^{\infty}\overline{B}(q_i,2^{-i})$ (or a slightly 
larger open set) with large enough $k$ we have that the conclusion of Theorem~\ref{thm:main result} holds.

Denote the $2$-dimensional Lebesgue measure by $\mathcal L^2$. For every Lebesgue point $x\in \R^2\setminus E$ 
and every $r>0$ we have
\[
0<\frac{\mathcal L^2(B(x,r)\cap E)}{\mathcal L^2(B(x,r))}<1,
\]
and so $P(E,B(x,\lambda r)) >0$
by the Poincar\'e inequality~\eqref{eq:poincare for BV}. Now by~\eqref{eq:def of theta} 
we must have $\mathcal H(\partial^*E\cap B(x,\lambda r))>0$, and so $\overline{\partial^*E}=\R^2\setminus E$.
\end{example}

This example demonstrates that the measure theoretic boundary of a set of finite perimeter need not be closed, that it can be 
much smaller than the topological boundary, and that the conclusion of Theorem~\ref{thm:main result} can fail in a very 
large set if we choose $G=\emptyset$. However, from Theorem~\ref{thm:main result}, by removing a suitable set $G$ of 
small capacity, both the topological and measure theoretic boundaries of a set of finite perimeter become very reasonably behaved.
For $A,E\subset X$, let us denote by $\partial_A E$ the boundary of $E\cap A$ in the subspace topology of $A$.

\begin{proposition}\label{prop:topological and measure theoretic boundary}
Let $E\subset X$ be a set of finite perimeter. For $\eps>0$ let $G\subset X$ be an open set provided by 
Theorem~\ref{thm:main result}, with $\capa_1(G)<\eps$. Then 
\begin{equation}\label{eq:topological and measure theoretic boundary}
\partial_{X\setminus G}\mathcal{I}(E)\subset \partial^*E\setminus G,
\end{equation}
and both $\partial_{X\setminus G}\mathcal{I}(E)$ and $\partial^*E\setminus G$ are closed subsets of $X$. 
\end{proposition}

\begin{proof}
If $x\in \partial_{X\setminus G}\mathcal{I}(E)$, there are sequences $y_i$ in $\mathcal{I}(E)\setminus G$ converging to $x$ and 
$z_i$ in
$X\setminus(\mathcal{I}(E)\cup G)$ also converging to $x$. Set $u:=\chi_E$. Then $u^{\wedge}(y_i)=1$ and $u^{\wedge}(z_i)=0$ 
(note that we have either $z_i\in \mathcal{E}(E)$ or $z_i\in \partial^*E$). Thus by Theorem~\ref{thm:main result},
we must have $u^{\wedge}(x)=0$ and $u^{\vee}(x)=1$, that is, $x\in\partial^*E$.

Now we show that $\partial^*E\setminus G$ is closed in $X\setminus G$. 
If $x_i\in \partial^*E\setminus G$ with $x_i\to x\in X\setminus G$, then $u^{\wedge}(x_i)=0$ and $u^{\vee}(x_i)=1$ for 
all $i\in\N$, so again by Theorem~\ref{thm:main result} we have $u^{\wedge}(x)=0$ and $u^{\vee}(x)=0$. Since $G$ is 
open, the sets $\partial_{X\setminus G}\mathcal{I}(E)$ and $\partial^*E\setminus G$ are closed also in $X$.
\end{proof}

\begin{lemma}\label{lem:curves}
For $i\in\N$, let $G_i\subset X$ be a nested sequence of sets (that is, $G_{i+1}\subset G_i$)
with $\capa_1(G_i)<2^{-i}$. Let $\widehat{\Gamma}$ be the 
family of non-constant curves that intersect each $G_i$. Then $\Mod_1(\widehat{\Gamma})=0$.
\end{lemma}

\begin{proof}
We will use the following observation in this proof. By~\cite[Theorem~1.56]{BB}, every function in $N^{1,1}(X)$ is 
absolutely continuous on $1$-almost every curve in $X$. 

For each $i\in\N$, take $u_i\in N^{1,1}(X)$ such that $0\le u_i\le 1$ on $X$, $u_i\ge 1$ in $G_i$, and 
$\Vert u_i\Vert_{N^{1,1}(X)}<2^{-i}$. The sequence 
$\{\sum_{i=1}^j u_i\}_{j\in\N}$ is a Cauchy sequence in $N^{1,1}(X)$, and converges therefore to $u:=\sum_{i\in\N}u_i\in N^{1,1}(X)$
(a proof of the fact that $N^{1,1}(X)$ is a Banach space can be found in~\cite{S,BB}).
Because for each $i$ we have $G_{i+1}\subset G_i$, we know that
$u$ is not bounded on any of the curves in $\widehat{\Gamma}$, 
and it follows that $u$ is not absolutely continuous on any of those curves. Now by the observation above,
the desired conclusion follows.
\end{proof}

As a consequence of Proposition~\ref{prop:topological and measure theoretic boundary} and Lemma~\ref{lem:curves}, we have the following analog of the result used in the proof of Federer's theorem, in the metric setting.

\begin{corollary}\label{cor:Fed1}
Let $E\subset X$ be of finite perimeter. Let $\Gamma$ be the collection of curves $\gamma$ in $X$
such that there exist $t_0, t_1\in [0,\ell_{\gamma}]$ with $t_0<t_1$ and
either
\begin{enumerate}
\item $\gamma(t_0)\in\mathcal{I}(E)$, $\gamma(t_1)\in \mathcal{E}(E)$ and $\gamma([t_0,t_1])\cap\partial^*E$ is empty,  or
\item $\gamma(t_0)\in\mathcal{E}(E)$, $\gamma(t_1)\in \mathcal{I}(E)$ and $\gamma([t_0,t_1])\cap\partial^*E$ is empty.
\end{enumerate}
Then $\Mod_1(\Gamma)=0$.
\end{corollary}

\begin{proof}
Let $u:=\chi_E$. For each $i\in\N$, let $G_i$ be an open set with 
$\capa_1(G_i)<2^{-i}$, given by Theorem~\ref{thm:main result}. By replacing
each $G_i$ with $\bigcap_{j=1}^iG_j$, if necessary, we may assume that for each $i\in\N$, $G_{i+1}\subset G_i$.
Let $\gamma\in\Gamma$. If there exists $i\in\N$ such that $\gamma$ does not intersect $G_i$, then necessarily $\gamma([t_0,t_1])\cap\partial^*E\neq\emptyset$
according to~\eqref{eq:topological and measure theoretic boundary}. We conclude that $\gamma$ intersects 
each set $G_i$, that is, $\Gamma\subset\widehat{\Gamma}$, and $\Mod_1(\widehat{\Gamma})=0$ by Lemma~\ref{lem:curves}.
\end{proof}

Now we prove the following result that partially generalizes Federer's structure theorem to the metric setting. 

\begin{theorem}
Let $E\subset X$ be bounded and $\mu$-measurable. Then $E$ is of finite perimeter if and only if $\mathcal{H}(\partial^*E)$ is finite and
$E$ satisfies the conclusion of Corollary~\ref{cor:Fed1}. 
\end{theorem}

\begin{proof}
 One part of the claim follows directly from Corollary~\ref{cor:Fed1}. Thus it suffices to prove that if $E$ satisfies the
 conclusion of Corollary~\ref{cor:Fed1} and $\mathcal{H}(\partial^*E)<\infty$, then $E$ is of finite perimeter. To do so, it suffices
 to find an $L^1$-approximation of $\chi_E$ with $L^1$-bounded weak upper gradients.
 
 Since $\mathcal{H}(\partial^*E)<\infty$, for each $\eps>0$ we can find a cover of $\partial^*E$ by balls 
$B_i=B(x_i,r_i)$, $i\in\N$, with radius no more than $\eps$, such that 
 \[
   \sum_{i\in\N}\frac{\mu(B_i)}{r_i}\le \mathcal{H}(\partial^*E)+\eps.
 \]
 For each ball $B_i$ in the cover, we fix a $1/r_i$-Lipschitz function
 $u_i$ such that $0\le u_i\le 1$ on $X$, $u_i=1$ on $B_i$, and the support of $u_i$ is contained in $2\overline{B_i}$. Now let
 \[
   u_\eps(x):=\begin{cases} 1&\text{ if }x\in\mathcal{I}(E),\\
                                           \min\{1, \sum_{i\in\N} u_i(x)\} &\text{ otherwise.} \end{cases}
 \]
 Furthermore, let $v_\eps(x):=\min\{1,\sum_{i\in\N}u_i(x)\}$.
 Note that because $E$ is bounded, $u_\eps\in L^1(X)$. Set 
 \[
  g_\eps:=\sum_{i\in\N}\frac{1}{r_i}\chi_{2B_i}.
 \]
 Clearly $g_{\eps}$ is an upper gradient of $v_{\eps}$. We will show that $g_{\eps}$ is an upper gradient of $u_{\eps}$ as well.
Take a curve $\gamma\notin \Gamma$ with end points $x,y$, where $\Gamma$ was defined in Corollary~\ref{cor:Fed1}.
 If $x,y\in X\setminus\mathcal{I}(E)$, then
 \[
 |u_\eps(x)-u_\eps(y)|=|v_\eps(x)-v_\eps(y)|\le \int_\gamma g_\eps\, ds.
 \]
 If the end points $x,y$ both lie in $\mathcal{I}(E)$,
 then $u_\eps(x)=u_\eps(y)$, and hence the upper gradient inequality
 \begin{equation}\label{eq:checking the upper gradient inequality}
 |u_\eps(x)-u_\eps(y)|\le \int_\gamma g_\eps\, ds
 \end{equation}
 is satisfied.
 If $x\in\mathcal{I}(E)$ and 
 $y\in X\setminus[\mathcal{I}(E)\cup\bigcup_{i\in\N} 2B_i]\subset\mathcal{E}(E)$, 
 then $|u_\eps(x)-u_\eps(y)|=1$, and since $\gamma\notin\Gamma$, the curve $\gamma$ intersects $\partial^*E$, and so it intersects $B_j$ for some $j$ and also
 intersects the complement of $2B_j$. Thus
 \[
 \int_\gamma g_\eps\, ds\ge \frac{1}{r_j}\int_{\gamma}\chi_{2B_j}\,ds\ge \frac{r_j}{r_j}\ge 1.
 \]
So again the pair $u_\eps,g_\eps$ satisfies the upper gradient inequality~\eqref{eq:checking the upper gradient inequality}.

 Finally, if $x\in\mathcal{I}(E)$ and $y\in \bigcup_{i\in\N}2B_i\setminus\mathcal{I}(E)$, again since $\gamma\notin\Gamma$, there is some $t_0\in[0,\ell_{\gamma}]$ such that 
$\gamma(t_0)\in\partial^*E$, and thus $\gamma(t_0)\in B_j$ for some $j\in\N$. Note that $u_\eps(x)=u_\eps(\gamma(0))=1$,
 $u_\eps(\gamma(t_0))=v_\eps(\gamma(t_0))=1$, and $u_\eps(y)=v_\eps(y)$.
  It follows that
 \begin{align*}
 |u_\eps(x)-u_\eps(y)|&\le |u_\eps(\gamma(0))-u_\eps(\gamma(t_0))|+|u_\eps(\gamma(t_0))-u_\eps(\gamma(\ell_{\gamma}))|\\
   &=|v_\eps(\gamma(t_0))-v_\eps(\gamma(\ell_{\gamma}))|\le \int_\gamma g_\eps\, ds.
 \end{align*}
 Thus in all cases the pair $u_\eps,g_\eps$ satisfies the upper gradient inequality for
 $1$-almost every
 curve in $X$. Furthermore, 
 \[
 \int_X g_\eps\, d\mu\le \sum_{i\in\N}\frac{\mu(2B_i)}{r_i}\le C_d\sum_{i\in\N}\frac{\mu(B_i)}{r_i}
   \le C_d( \mathcal{H}(\partial^*E)+\eps)<\infty. 
 \]
 It follows that for $0<\eps\le 1$, $u_\eps\in N^{1,1}(X)$ with $1$-weak upper gradients $g_\eps$ with a bounded $L^1$-norm. Moreover, 
 \begin{align*}
   \int_X|u_\eps-\chi_E|\, d\mu
   &\le \int_X \chi_{\bigcup_{i\in\N}2B_i}\, d\mu\le \sum_{i\in\N}\mu(2B_i)\le \eps\sum_{i\in\N} \frac{\mu(2B_i)}{r_i}\\
   &\le \eps (\mathcal{H}(\partial^*E)+1)\to 0
 \end{align*}
 as $\eps\to 0$.
 It follows that $u_\eps\to \chi_E$ in $L^1(X)$, and thus $\chi_E\in\BV(X)$, that is, $E$ is of finite perimeter.
\end{proof}

\section{Strong quasicontinuity}

It is known that if the measure on a metric measure space $X$ is doubling and supports a 
$(1,p)$-Poincar\'e inequality
for some $1\le p<\infty$, then Lipschitz functions are dense in $N^{1,p}(X)$, see for 
example~\cite[Theorem~8.2.1]{HKST}. Similarly, Lemma~\ref{lem:density of continuous functions} shows 
that continuous functions are dense in the space of $\BV$ functions with a $\mathcal H$-negligible jump set. 
On the other hand, from Proposition~\ref{prop:quasicontinuity without jump part} we know that the restrictions of $\BV$ functions 
with a $\mathcal H$-negligible jump set outside sets of small capacity are continuous, just like the restrictions of Newton-Sobolev functions.

The concept of \emph{strong quasicontinuity} essentially combines these two results: it involves a Lusin-type 
approximation of a function $u$ by a continuous function that approximates $u$ simultaneously in the BV (or Newton-Sobolev) 
norm~\emph{and} outside a set of small capacity.
In~\cite[Theorem 7.1]{KKST} such a Lusin-type approximation result for Newton-Sobolev functions was given. 
Here we show strong quasicontinuity for $\BV$ functions with a $\mathcal H$-negligible jump set. Note that such $\BV$ functions
need not be in the Newton-Sobolev class, since the Cantor part of their variation measure need not be zero. 

\begin{lemma}\label{lem:uniform-fracMaxgradmeas}
Let $u\in\BV(X)$ with $\mathcal H(S_u)=0$, and let $\eps>0$. Then there exists 
$G\subset X$ with $\capa_1(G)<\eps$ such that
\[
r\frac{\Vert Du\Vert(B(x,r))}{\mu(B(x,r))}\to 0\qquad\textrm{as }r\to 0
\]
uniformly in $X\setminus G$.
\end{lemma}

\begin{proof}
Given $\delta>0$, let
\[
A:=\left\{x\in X:\, \limsup_{r\to 0^+}r\frac{\Vert Du\Vert(B(x,r))}{\mu(B(x,r))}
\ge \delta\right\}.
\]
By~\cite[Theorem 2.4.3]{AT}, 
we know that $\Vert Du\Vert(A)\ge \delta \mathcal H(A)$. 
Now by~\cite[Theorem~5.3]{AMP}, since $\Vert Du\Vert^j(X)=0$,
we have $\Vert Du\Vert(F)=0$ for any $F$ with $\mathcal H(F)<\infty$, 
and so we must have $\Vert Du\Vert(A)=\mathcal H(A)=0$. 
It follows that
\[
  \mathcal{H}\left(\bigg\lbrace x\in X :\, \limsup_{r\to 0^+}r\frac{\Vert Du\Vert(B(x,r))}{\mu(B(x,r))}>0\bigg\rbrace\right)=0.
\]

By Egorov's theorem, we can pick compact sets $H_1\subset H_2\subset \ldots$ and radii
$1/5\ge r_1\ge r_2\ge \ldots> 0$ such that $\Vert Du \Vert(X\setminus H_i)<2^{-i}\eps$, and
\[
r\frac{\Vert Du\Vert(B(x,r))}{\mu(B(x,r))}\le \frac{1}{C_d^2 i}
\]
for all $x\in H_i$ and $r\in (0,2r_i]$. Then define for $i\in\N$
\begin{align*}
&G_i:=\\
&\left\{x\in X\setminus H_i:\exists r\in (0,r_i]\ \textrm{s.t. }B(x,r)\subset X\setminus H_i\  
  \textrm{and}\  r\frac{\Vert Du\Vert(B(x,r))}{\mu(B(x,r))}> \frac 1i\right\}.
\end{align*}
Now we show that for all $x\in X\setminus G_i$ and $r\in (0,r_i]$,
\begin{equation}\label{eq:uniform convergence first step}
r\frac{\Vert Du\Vert(B(x,r))}{\mu(B(x,r))}\le \frac 1i.
\end{equation}
The only case that needs to be checked is when $x\in X\setminus (H_i\cup G_i)$ and 
$B(x,r)\cap H_i\neq \emptyset$ for some $r\in (0,r_i]$. Then for any point $y\in B(x,r)\cap H_i$, we have
\[
r\frac{\Vert Du\Vert(B(x,r))}{\mu(B(x,r))}\le C_d^2 r\frac{\Vert Du\Vert(B(y,2r))}{\mu(B(y,2r))}\le \frac 1i
\]
by the definition of the sets $H_i$.

Fix $i\in\N$. From the definition of $G_i$ we get a covering $\{B(x,r(x))\}_{x\in G_i}$ of $G_i$, and by the $5$-covering 
theorem, we obtain a countable collection of disjoint balls $\{B(x_j,r_j)\}_{j\in\N}$ such that the balls $B(x_j,5r_j)$ cover $G_i$.
Thus
\begin{align*}
\frac {1}{2C_d} \capa_1(G_i)\le \mathcal H_{1}(G_i)
&\le \sum_{j\in\N}\frac{\mu(B(x_j,5r_j))}{5r_j}
\le C_d^3\sum_{j\in\N}\frac{\mu(B(x_j,r_j))}{r_j}\\
&\le C_d^3 i\sum_{j\in\N}\Vert Du\Vert(B(x_j,r_j))
\le C_d^3 i\Vert Du\Vert(X\setminus H_i).
\end{align*}
Let $G:=\bigcup_{i\in\N}G_i$, so that
\[
\capa_1(G)\le \sum_{i\in\N}\capa_1(G_i)\le C\sum_{i\in\N}i\Vert Du\Vert(X\setminus H_i)\le C\sum_{i\in\N}i 2^{-i}\eps \le C\eps.
\]
Moreover, by~\eqref{eq:uniform convergence first step}, for every $x\in X\setminus G$, $i\in \N$, and $r\in (0,r_i]$ 
we have
\[
r\frac{\Vert Du\Vert(B(x,r))}{\mu(B(x,r))}\le \frac 1i.
\]
\end{proof}

\begin{proposition}\label{prop:uniform Lebesgue points without jump set}
Let $u\in\BV(X)$ with $\mathcal H(S_u)=0$, and let $\eps>0$. Then there exists
$G\subset X$ with $\capa_1(G)<\eps$ such that
\[
\vint{B(x,r)}|u-\widetilde{u}(x)|\,d\mu\to 0\qquad\textrm{as }r\to 0
\]
locally uniformly in the set $X\setminus G$.
\end{proposition}

\begin{proof}
Our proof largely follows corresponding proofs concerning Lebesgue points of Newton-Sobolev functions, see 
e.g.~\cite[Theorem 4.1]{KKST3}. First note that
\begin{align*}
\vint{B(x,r)}|u-\widetilde{u}(x)|\,d\mu
&\le \vint{B(x,r)}|u-u_{B(x,r)}|\,d\mu+|u_{B(x,r)}-\widetilde{u}(x)|\\
&\le Cr\frac{\Vert Du\Vert(B(x,\lambda r))}{\mu(B(x,\lambda r))}+|u_{B(x,r)}-\widetilde{u}(x)|.
\end{align*}
The first term converges uniformly to zero as $r\to 0$ outside a set $F$ with
$\capa_1(F)<\eps/2$ by Lemma~\ref{lem:uniform-fracMaxgradmeas}. 
So we only need to consider the second term. By Lemma~\ref{lem:density of continuous functions}, there 
is a sequence $u_i\in \BV(X)\cap C(X)$ with 
\[
\Vert u_i-u\Vert_{\BV(X)}\le \frac{2^{-2i-2}\gamma}{C_1}\eps,
\]
where $C_1$ is the constant from Lemma~\ref{lem:capacity level set maximal function}, corresponding to the choice $R=1$.
For $i\in\N$, let
\[
G_i:=\{x\in X:\, \max \{|u_i(x)-\widetilde{u}(x)|,\ \mathcal M_1 (u_i-u)(x)\}> 2^{-i}\}.
\]
By Lemma~\ref{lem:capacity level set maximal function} and the proof of Lemma~\ref{lem:from BV to uniform convergence}, 
$\capa_1(G_i)\le 2^{-i-1}\eps$. Define $G:=\bigcup_{i\in\N} G_i\cup F$, so that $\capa_1(G)< \eps$. Now for 
$x\in X\setminus G$ and $r\in (0,1]$,
\begin{align*}
|u_{B(x,r)}-\widetilde{u}(x)|&\le |u_{B(x,r)}-(u_i)_{B(x,r)}|+|(u_i)_{B(x,r)}-u_i(x)|
+|u_i(x)-\widetilde{u}(x)|\\
&\le \mathcal M_1(u_i-u)(x)+|(u_i)_{B(x,r)}-u_i(x)|+|u_i(x)-\widetilde{u}(x)|\\
&\le 2^{-i}+|(u_i)_{B(x,r)}-u_i(x)|+2^{-i}\\
&\le 2^{-i+1}+\vint{B(x,r)}|u_i-u_i(x)|\,d\mu.
\end{align*}
Fix a ball $B(z,\widetilde{r})$, and $\delta >0$. Picking $i$ sufficiently large, the first term above is 
less than $\delta/2$. Then the corresponding function $u_i$ is, as a continuous function, locally uniformly 
continuous, so that it is uniformly continuous in $B(z,\widetilde{r}+1)$. Thus we can pick $r>0$ small 
enough that the second term is less than
$\delta/2$ for every $x\in B(z,\widetilde{r})$. Since $\delta>0$ was arbitrary, this establishes local uniform convergence.
\end{proof}

\begin{theorem}\label{thm:strong quasicontinuity without jump part}
Let $u\in\BV(X)$ with $\mathcal H(S_u)=0$, and let $\eps>0$. Then there exists
an open set
$G\subset X$ with $\capa_1 (G)<\eps$, and $w\in \BV(X)\cap C(X)$ such that $w=\widetilde{u}$ in $X\setminus G$ and 
$\Vert w-u\Vert_{\BV(X)}<\eps$.
\end{theorem}
\begin{proof}
By Lusin's and Egorov's theorems, we can find
an open set
$F\subset X$ with $\Vert Du\Vert(F)<\eps$ such that
$\widetilde{u}|_{X\setminus F}$ is continuous and
\begin{equation}\label{eq:unif conv for Lusin type approx}
\vint{B(x,r)}|u-\widetilde{u}(x)|\,d\mu\to 0\qquad\textrm{as }r\to 0
\end{equation}
uniformly in the set $X\setminus F$.
By Theorem~\ref{thm:main result} and Proposition~\ref{prop:uniform Lebesgue points without jump set}
and the fact that $\capa_1$ is an outer capacity,
we can find
an open set
$\widetilde{G}\subset X$ with $\capa_1 (\widetilde{G})<\eps$ such that 
$\widetilde{u}|_{X\setminus \widetilde{G}}$ is continuous and the convergence in~\eqref{eq:unif conv for Lusin type approx} is
locally uniform in the set $X\setminus \widetilde{G}$. Defining $G:=\widetilde{G}\cap F$, we have $\capa_1(G)<\eps$ and 
$\Vert Du\Vert(G)<\eps$. Apply Corollary~\ref{cor:mollifying in an open set} with $U=G$ and $\kappa=\eps$ to obtain 
a function $w\in\BV(X)$ with $\Vert w-u\Vert_{\BV(X)}\le C\eps$. Then by 
Proposition~\ref{prop:uniform convergence and continuity}, $w\in C(X)$ and $w=\widetilde{u}$ in $X\setminus G$.
\end{proof}

We say that $X$ supports a \emph{strong relative isoperimetric inequality} if for every $\mu$-measurable set $E\subset X$,
$P(E,X)<\infty$ whenever $\mathcal{H}(\partial^*E)<\infty$, see the discussion in 
Section~\ref{sec:sets of finite perimeter} as well as~\cite{KKST} and~\cite{KLS} for more on this question.
In~\cite[Theorem 7.1]{KKST} the following Lusin-type approximation for Newton-Sobolev functions was given. The 
authors made the additional assumption 
that the space supports a strong relative isoperimetric inequality, which we can now remove.

\begin{corollary}\label{rmk:strong quasicontinuity}
Let $1\le p<\infty$, $u\in N^{1,p}(X)$, and $\eps>0$. Then there exists an open set $G\subset X$ and $w\in N^{1,p}(X)\cap C(X)$ such
that $\capa_p(G)<\eps$, $w=\widetilde{u}$ in $X\setminus G$, and $\Vert w-u\Vert_{N^{1,p}(X)}<\eps$.
\end{corollary}
\begin{proof}
When $p=1$, this is a special case of Theorem~\ref{thm:strong quasicontinuity without jump part},
since $\Vert w-u\Vert_{N^{1,1}(X)}\le C\Vert w-u\Vert_{\BV(X)}$, see \cite[Theorem 4.6]{HKLL}.
The case $1<p<\infty$ 
follows by suitably adapting
Theorem~\ref{thm:zero extension} (see~\cite[Theorem~1.1]{KKST}),
Theorem~\ref{thm:mollifying in an open set}, Proposition~\ref{prop:uniform convergence and continuity} (the same proof applies), 
and Proposition~\ref{prop:uniform Lebesgue points without jump set}, combined with the $p$-quasicontinuity of $u\in N^{1,p}(X)$.
\end{proof}

In this section so far, we have only dealt with $\BV$ functions with a $\mathcal H$-negligible jump set. A strong version of 
our quasicontinuity-type result, Theorem~\ref{thm:main result}, would be the following. Note that below we 
require~\eqref{eq:strong quasicontinuity} to hold \emph{everywhere}, not just outside a set of small capacity.

\begin{openproblem}\label{conj:strong quasicontinuity}
Let $u\in\BV(X)$ and let $\eps>0$. Then there exists an open set $G\subset X$ with
$\capa_1(G)<\eps$, and $w\in \BV(X)$ such that $w^l=u^l$ in $X\setminus G$ for all $l=1,\ldots,n$, 
$\Vert w-u\Vert_{\BV(X)}<\eps$, and whenever $y_k\to x\in X$,
\begin{equation}\label{eq:strong quasicontinuity}
\min_{l_2\in\{1,\ldots,n\}} |w^{l_1}(y_k)-w^{l_2}(x)|\to 0
\end{equation}
for each $l_1=1,\ldots,n$.
\end{openproblem}

Though we can pick a set $G$ as in Theorem~\ref{thm:main result}, it is not obvious how the function $w$ should be defined in $G$ to ensure that~\eqref{eq:strong quasicontinuity} holds.
On the other hand, we do get the following Lusin-type approximation for general $\BV$ functions.

\begin{theorem}\label{thm:lusin-type approximation}
Let $u\in\BV(X)$ and $\eps>0$. Then for any open set $W\supset S_u$
there exists an open set $V\supset W$ with $\capa_1(V\setminus W)<\eps$, and
a function $v\in\BV(X)\cap C(X)$ with $v=\widetilde{u}$ in $X\setminus V$ and
\begin{equation}\label{eq:Lusin bounds in jump set}
\Vert v-u\Vert_{L^1(X)}\le \eps, \qquad\ \Vert D(v-u)\Vert(X)\le C\Vert Du\Vert^j(X)+\eps.
\end{equation}
\end{theorem}

For example, we can require $W$ and hence $V$ to have $\mu$-measure less than $\eps$. This theorem 
also gives better control of $\Vert D(v-u)\Vert(X)$ than a 
Lusin-type approximation by a Lipschitz function given in \cite[Proposition 4.3]{KKST2}, but on the downside, 
we only get an approximation by a \emph{continuous} function.

\begin{proof}
By making $W$ smaller, if necessary, we can assume that $\Vert Du\Vert(W)\le \Vert Du\Vert(S_u)+\eps$.
Apply Corollary~\ref{cor:mollifying in an open set} with $U=W$ and 
$\kappa=\eps/2$ to obtain a function $w\in\BV(X)$ with $w=u$ in $X\setminus W$, $\Vert w-u\Vert_{L^1(X)}\le \eps/2$, and
\[
\Vert D(w-u)\Vert(X)=\Vert D(w-u)\Vert(W)\le C\Vert Du\Vert(W)\le
C\Vert Du\Vert(S_u)+C\eps.
\]
Note that by~\eqref{eq:zero boundary value for mollification}, we have in fact
$w^{\wedge}=w^{\vee}=\widetilde{u}$ in 
$X\setminus (W\cup \widetilde{N})$, for some $\mathcal H$-negligible set $\widetilde{N}\subset X$. By Remark \ref{rmk:capacities etc}, there exists an open set $N\supset \widetilde{N}$ with $\capa_1(N)<\eps/2$.
Furthermore, $\mathcal H(S_w)=0$, so that we can apply 
Theorem~\ref{thm:strong quasicontinuity without jump part}
to get an open set $G\subset X$ with 
$\capa_1 (G)<\eps/2$ and a function $v\in \BV(X)\cap C(X)$ with $v=\widetilde{w}$ in $X\setminus G$ and 
$\Vert v-w\Vert_{\BV(X)}\le \eps/2$. Thus for $V:= W\cup N\cup G$ we have $v=\widetilde{u}$ in $X\setminus V$, $\Vert v-u\Vert_{L^1(X)}\le \eps$, and
\[
\Vert D(v-u)\Vert(X)\le C\Vert Du\Vert(S_u)+C\eps=C\Vert Du\Vert^j(X)+C\eps.
\]
\end{proof}

If $X$ supports a strong relative isoperimetric inequality, we can use 
the proposition below instead of Corollary~\ref{cor:mollifying in an open set} in 
the proof of Theorem~\ref{thm:lusin-type approximation}, and then we will
get~\eqref{eq:Lusin bounds in jump set} with the constant $C=2+\eps$.

\begin{proposition}[{\cite[Corollary 6.7]{LS}}]
Let $U\subset X$ be an open set, and let $u\in\BV(U)$. Assume either that the space supports a strong 
relative isoperimetric inequality, or that $\mathcal H(\partial U)<\infty$. Then there exist functions
$\breve{v}_i\in \liploc(U)$, $i\in\N$, with $\breve{v}_i\to u$ in $L^1(U)$, 
$\Vert D\breve{v}_i\Vert(U)\to \Vert Du\Vert(U)$, and such that the functions
\[
h_i:=
\begin{cases}
\breve{v}_i-u &\ \text{in }U,\\
0  &\ \text{in }X\setminus U,
\end{cases}
\]
satisfy $h_i\in \BV(X)$ with $\Vert D h_i\Vert(X\setminus U)=0$.
\end{proposition}

\end{document}